\documentclass[12pt,A4,leqno]{amsart}
\usepackage{amsfonts,bm}
\usepackage{mathrsfs}
\usepackage[T1]{fontenc}
\usepackage{amssymb}
\usepackage{amsmath,amscd}
\usepackage{color}
\usepackage{epsf}
\usepackage{graphicx}
\usepackage{xypic}
\usepackage{extarrows}
\usepackage{hyperref} 
\hypersetup{colorlinks,linkcolor={blue},citecolor={blue}} 
\usepackage[capitalise]{cleveref} 
\usepackage{booktabs}
\usepackage{multirow}
\usepackage{caption}
\usepackage{longtable}

\theoremstyle{plain}
\newtheorem{thm}{Theorem}[section]
\newtheorem{lem}[thm]{Lemma}

\newtheorem{cor}[thm]{Corollary}

\theoremstyle{definition}
\newtheorem{defi}[thm]{Definition}
\newtheorem{rem}[thm]{Remark}

\newcommand{\Z}{\mathbb Z}
\newcommand{\F}{\mathbb F}
\newcommand{\nn}{\vskip 0.2cm}
\newcommand{\n}{\vskip 0.1cm}

\newcommand{\ZZ}{\mathcal Z}
\newcommand{\link}{\mathrm{link}}
\renewcommand{\star}{\mathrm{star}}
\newcommand{\mdim}{\mathrm{mdim}}
\renewcommand{\geq}{\geqslant}
\renewcommand{\leq}{\leqslant}
\renewcommand{\emptyset}{\varnothing}

\newcommand{\wt}{\widetilde}
\renewcommand{\subset}{\subseteq}

\begin{document}

\title [\ ] {On Stanley-Reisner Rings with Minimal 
Betti Numbers}

\author{Pimeng Dai*, \  Li Yu**}
\address{*School of Mathematics, Nanjing University, Nanjing, 210093,
P.R.China}
 \email{pimengdai@gmail.com}
\address{**School of Mathematics, Nanjing University, Nanjing, 210093,
P.R.China}
 \email{yuli@nju.edu.cn}




\begin{abstract}
We study simplicial complexes with a given number of vertices whose Stanley-Reisner ring has the minimal sum of Betti numbers. We prove that every full subcomplex of them is homotopy equivalent either to a point or to a sphere. Moreover, we prove that
such kind of simplicial complexes coincide with those having minimal Taylor resolutions. 
 \end{abstract}

\maketitle

\section{Introduction}
 Let $K$ be a simplicial complex whose vertex set
  is denoted by 
   $$V(K)=[m]=\{v_1,\cdots, v_m\}.$$  The \emph{Stanley-Reisner ring} of $K$ over a field
$\mathbb{F}$ (see Stanley~\cite{Stanley07}) is 
\[ \mathbb{F}[K]= \mathbb{F}[v_1,\cdots,v_m] \slash \mathcal{I}_K\]
where $\mathcal{I}_K$ is the ideal in the polynomial ring $\mathbb{F}[v_1,\cdots,v_m]$ generated by all the squarefree monomials
$v_{i_1}\cdots v_{i_s}$ where $\{v_{i_1},\cdots , v_{i_s}\}$ is not a simplex of $K$.  The natural 
$\mathbb{F}[v_1,\cdots,v_m]$-module structure of $\mathbb{F}[K]$ determines a canonical 
algebra $\mathrm{Tor}_{\mathbb{F}[v_1,\cdots, v_m]}(\mathbb{F}[K],\mathbb{F})$, where $\mathbb{F}$
is considered as a trivial $\mathbb{F}[v_1,\cdots, v_m]$-module.
Moreover, there is a bigraded structure on $\mathrm{Tor}_{\mathbb{F}[v_1,\cdots, v_m]}(\mathbb{F}[K],\mathbb{F})$ :   
  \[ \mathrm{Tor}_{\mathbb{F}[v_1,\cdots, v_m]}(\mathbb{F}[K],\mathbb{F}) = \bigoplus_{i,j\geq 0} \mathrm{Tor}^{-i,2j}_{\mathbb{F}[v_1,\cdots, v_m]}(\mathbb{F}[K],\mathbb{F}) \] 
  where $\mathrm{deg}(v_{\ell})=2$ for each $1\leqslant \ell \leqslant m$. The integers 
  \begin{equation} \label{Equ:Def-Bigraded-Betti}
   \beta^{-i,2j}(\mathbb{F}[K]) := \dim_{\mathbb{F}}\mathrm{Tor}^{-i,2j}_{\mathbb{F}[v_1,\cdots, v_m]}(\mathbb{F}[K],\mathbb{F})
   \end{equation}
  are called
the \emph{bigraded Betti numbers} of $\mathbb{F}[K]$.

   \begin{defi} \label{Def:Bigraded-Betti}
   The \emph{$i$-th Betti number} of the Stanley-Reisner ring $\mathbb{F}[K]$ is
   \begin{equation} \label{Equ:Def-Betti}
   \beta^{-i}(\mathbb{F}[K]) = \sum_{j\geq 0} \beta^{-i, 2 j}(\mathbb{F}[K]).
   \end{equation}    
   The \emph{total Betti number} of $\mathbb{F}[K]$ is
       \begin{equation} \label{Equ:Def-Total-Betti}
         \widetilde{D}(K;\mathbb{F}) = \sum_{i\geq 0}  \beta^{-i}(\mathbb{F}[K])  = \sum_{i,j\geq 0} \beta^{-i,2j}(\mathbb{F}[K]) = \dim_{\mathbb{F}} \mathrm{Tor}_{\mathbb{F}[v_1,\cdots, v_m]}(\mathbb{F}[K],\mathbb{F}). 
         \end{equation}
  \end{defi}
   \begin{rem}
      In many literatures, people also use 
$\beta_{i,j}(\mathbb{F}[K])$ and $\beta_i(\mathbb{F}[K])$ to refer to bigraded Betti numbers and Betti numbers of $\mathbb{F}[K]$, respectively. Our notations in~\eqref{Equ:Def-Bigraded-Betti} and~\eqref{Equ:Def-Betti} are taken from toric topology (see Buchstaber and Panov~\cite{BP15}) which avoid confusions with the notations for the ordinary Betti numbers of $K$.
\end{rem}

The reader is referred to~\cite{Stanley07, MillSturm04, EnHerz11, Villa15} for more information of Stanley-Reisner rings and their role in combinatorial commutative algebra.\n

In this paper, we study simplicial complexes whose Stanley-Reisner rings have the smallest possible total Betti number. People have studied the Betti numbers $\beta^{-i}(\mathbb F[K])$
in many circumstances and obtain different forms of 
lower bounds. The following is the most relevant one to our study:
\begin{equation}\label{Thm:Main-Binom}
	\beta^{-i}(\mathbb F[K])
	\geq
	\binom{m-\mdim(K)-1}{i},
\end{equation}
where ``$\mdim(K)$'' here denotes the minimum dimension among all  the maximal simplices of $K$.
It follows that
\begin{equation}\label{Equ:uto-mdim}
	\widetilde D(K;\mathbb F)
	\geq
	2^{m-\mdim(K)-1}.
\end{equation}

The lower bound of $\beta^{-i}(\mathbb F[K])$ in~\eqref{Thm:Main-Binom} was not formally written in the literature before. It was implicitly stated in the remark
at the end of Charalambous~\cite{Char91}, and can also be derived from the work of Evans and Griffith~\cite{EvanGriff88} and Morey and Villarreal~\cite{MorVilla12}. Here we will give a new proof of~\eqref{Thm:Main-Binom} in Section~\ref{Sec:Lower-Bound}. 
Note that the inequality~\eqref{Equ:uto-mdim} was also obtained by Ustinovskii~\cite{Uto11} via a topological method. These lower bounds motivate us to define the following notion.

\begin{defi} \label{Def:Taut}
	A simplicial complex $K$ with $m$ vertices is called
	\emph{taut} over $\mathbb F$ if
	\[
	\widetilde D(K;\mathbb F)
	=
	2^{m-\mdim(K)-1}.
	\]
\end{defi}

Equivalently, $K$ is taut if and only if for every
$0\leq i \leq m-\mdim(K)-1$,
\[
\beta^{-i}(\mathbb F[K])
= \binom{m-\mdim(K)-1}{i}.
\]
 
The main purpose of this paper is to characterize taut
simplicial complexes and their Stanley-Reisner rings from
algebraic, topological and combinatorial perspectives.  
We will prove that any taut simplicial complex $K$ is homotopy equivalent either to a point or to a sphere (Theorem~\ref{Thm:Homotopy-Type}), and every full subcomplex of $K$ is taut (Theorem~\ref{Thm:Subcomplex-Taut}). These results lead to a  topological characterization of taut simplicial complexes in Corollary~\ref{Cor:Character}. 
In particular, we find that the tautness of a simplicial complex is independent on the coefficient $\mathbb{F}$.
\n 

In addition, we establish an algebraic-combinatorial characterization
of taut simplicial complexes $K$ through the Taylor resolution of $\mathbb{F}[K]$.  The following is the main theorem of the paper.

\begin{thm}[{Theorem~\ref{Thm:Taylor-2}}] \label{Thm:Taylor}
	For a simplicial complex $K$ with $m$ vertices, the following conditions are
	equivalent:
	\begin{itemize}
	 \item[(a)] $K$ is taut.\n
 \item[(b)] The number of minimal non-faces of $K$ is equal to $m-\operatorname{mdim}(K)-1$.\n
 \item[(c)] $K$ has a minimal Taylor resolution.
	\end{itemize}
\end{thm}

Recall that
  a non-empty subset $J$ of $V(K)$ is called a \emph{minimal non-face} of $K$ if $J$ is not a face of $K$ while every proper subset of $J$ is a face of $K$.   As a corollary of this theorem, we classify all the taut simplicial complexes that are Cohen-Macaulay (see Corollary~\ref{Cor:Main-1-Repeat}) or those that are shifted
(see Theorem~\ref{Thm:Taut-Shifted-Complex}).\n

Furthermore, we study combinatorial operations on simplicial complexes that may preserve tautness; for example: join and simplicial wedge (see Theorem~\ref{Thm:wedge-Weak-Taut}). \n

The paper is organized as follows. In Section~\ref{Sec:Prelim}, we introduce some basic notations. In Section~\ref{Sec-D}, we discuss the properties of $\widetilde D(K)$. In Section~\ref{Sec:Lower-Bound}, we  establish some lower bounds for the Betti numbers and bigraded Betti numbers of $\mathbb{F}[K]$. In Section~\ref{Sec-Weakly-Taut}, we investigate the topological structures of taut simplicial complexes. In Section~\ref{Sec:Taylor}, we study the relation between tautness of simplicial complexes and minimal Taylor resolution and prove Theorem~\ref{Thm:Taylor}.
In Section~\ref{Sec:Shifted}, we obtain a combinatorial characterization of all the shifted taut complexes. In Section~\ref{Sec:Wedge}, we discuss some combinatorial operations on simplicial complexes that preserve tautness.
 \vskip .4cm
 
 \section{Preliminary} \label{Sec:Prelim}
  For a finite CW-complex $X$ and a field $\F$, 
  let
\[ \beta_i(X;\F) := \dim_{\F} H_i(X;\F), \ \ \ \widetilde{\beta}_i(X;\F) := \dim_{\F} \widetilde{H}_i(X;\F)\]  
   where $H_i(X;\F)$ and $\widetilde{H}_i(X;\F)$ are the singular and reduced singular homology group of $X$ with $\F$-coefficients, respectively. Moreover, 
  define
   \[ tb(X;\F):= \sum_{i} \beta_i(X;\F), \ \ \  \widetilde{tb}(X;\F):= \sum_{i} \widetilde{\beta}_i(X;\F). \]  
    We call $tb(X;\F)$ and $\widetilde{tb}(X;\F)$ the  \emph{total Betti number} and \emph{reduced total Betti number} of $X$ with $\F$-coefficients, respectively. 
    The difference between $tb(X;\F)$ and $\widetilde{tb}(X;\F)$ is just $1$. 
  \n
    
    Let $K$ be a simplicial complex whose vertex set
  is $V(K)=[m]=\{v_1,\cdots, v_m\}$. For brevity, we will simply say that $K$ is a simplicial complex on $[m]$ in the rest of the paper.
   Each simplex
    $\sigma$ of $K$ is considered as a subset of $[m]$ whose dimension
   $$\dim(\sigma)=|\sigma|-1,$$ 
   where $|\sigma|$ is the cardinality of $\sigma$. We also use \emph{face} to refer to a simplex of $K$.
   \n
   
 The set of simplices of $K$ are ordered by inclusions. So every simplex of $K$ is included in some maximal simplex (may not be unique). For convenience, we introduce the following notations for our discussion:
   \begin{align*}
   \|K\| &= \text{the geometric realization of}\ K;\\
      \Delta^{max}(K) &= \text{the set of all maximal simplices of $K$};\\
      \mathrm{mdim}(K)&=\text{the minimal dimension of the maximal simplices of $K$};\\
      \Delta^{max}_{\mdim}(K) &= \{\xi\in \Delta^{max}(K)\mid  \dim(\xi) = \mathrm{mdim}(K)\} \subseteq  \Delta^{max}(K) ;  \\
      V_{\mdim}(K) &= \{ v\in V(K)\mid v\ \text{is a vertex of some simplex in}\ \Delta^{max}_{\mdim}(K)\}.
  \end{align*}
    \n

    Note that the difference between $\dim(K)$ and $\mdim(K)$ could be very large.
  If $\mdim(K)=\dim(K)$, then all the maximal simplices of $K$ have the same dimension, which means that $K$ is a \emph{pure} simplicial complex.
   \n
   
   For any simplex $\sigma$ of $K$, the \emph{link} and the \emph{star} of $\sigma$ are the subcomplexes
$$
\begin{aligned}
& \link_K \sigma=\{\tau \in K \mid \sigma \cup \tau \in K, \sigma \cap \tau=\varnothing\} ; \\
& \star_K \sigma=\{\tau \in K \mid \sigma \cup \tau \in K\} .
\end{aligned}
$$
   
   For any subset $J\subseteq  [m]$, let 
    \[  K|_J := \ \text{the \emph{full subcomplex} of $K$ obtained by restricting to}\ J.\]

   When $J=\varnothing$, $K|_J$ is the irrelevant complex $\{\varnothing\}$. Define\n
  \begin{itemize}
     \item $\beta_{i}(\{\varnothing\};\mathbb{F})=0, \ \text{for all}\ i\geqslant 0;\ \ \
    \widetilde{\beta}_{i}(\{\varnothing\};\mathbb{F})=  \begin{cases}
   1 ,  &  \text{if $i= -1$ }; \\
   0,  &  \text{otherwise}.
 \end{cases} $
 \end{itemize}
   \n
  
    For a vertex $v$ of $K$, we also use the following notation for brevity:
   $$K\setminus v := K|_{V(K)\setminus \{v\}}, \ \ \
    m_v := |V(\link_K v)|.$$
  
 \vskip .4cm
 
 \section{Properties of $\widetilde{D}(K)$} \label{Sec-D}
     
The following formula due to Hochster~\cite{Hoc77} (also see~\cite[Theorem 4.8]{Stanley07}) tells us that the bigraded Betti numbers of a simplicial complex $K$ can be computed from the homology groups of the full subcomplexes of $K$.

\begin{thm}[Hochster's formula] \label{Thm:Hoch}
 Let $K$ be a simplicial complex on $[m]$. 
 Then the bigraded Betti numbers of $\mathbb{F}[K]$ can be computed as 
 \begin{equation}\label{Equ:Hochster}
   \beta^{-i,2j}(\mathbb{F}[K]) = \sum_{J\subseteq [m], |J|=j} 
   \dim_{\mathbb{F}} \widetilde{H}_{j-i-1}(K|_J;\mathbb{F}) =\sum_{J\subseteq [m], |J|=j} \widetilde{\beta}_{j-i-1}(K|_J;\mathbb{F}).
 \end{equation} 
\end{thm}
  By this formula, we can immediately see the following facts:\n
  \begin{itemize}
   \item $\beta^{-i,2j}(\mathbb{F}[K*\Delta^{[r]}])=\beta^{-i,2j}(\mathbb{F}[K])$ for any $r\geq 1$.\n
  \item $\beta^{-i,2j}(\mathbb{F}[K])=0$ if $j\leq i$ or if $j>m$ or $i>m$.\n
  \item $\beta^{0,2j}(\mathbb{F}[K])=0$ for any $j> 0$, and $\beta^{0,0}(\mathbb{F}[K])=1$.
  \end{itemize}
  So the calculation of $ \beta^{-i,2j}(\mathbb{F}[K]) $ is nontrivial only when $1\leq i< j \leq m$.
  \n
 By the Hochster's formula, we can express $\widetilde{D}(K;\mathbb{F})$ as
    \begin{equation} \label{Equ:D-tilde-K-J}
    \widetilde{D}(K;\mathbb{F})= \sum_{J\subseteq [m]} \widetilde{tb}(K|_J;\mathbb{F}). 
    \end{equation}
    This explains why we put ``$\, \widetilde{\ }\, $'' in the notation
    $\widetilde{D}(K;\mathbb{F})$.\n
 
 There is another interpretation of
  $\widetilde{D}(K;\mathbb{F})$ from a canonical CW-complex $\ZZ_K$ associated to $K$ called the \emph{moment-angle complex} of $K$ (see Davis and Januszkiewicz~\cite[pp.\,428--429]{DaJan91} or Buchstaber and Panov~\cite[Section 4.1]{BP15}). One way to write $\ZZ_K$ is 
  \begin{equation} \label{Equ:ZK} 
 \ZZ_K =  \bigcup_{\sigma\in K} \Big( \prod_{v_i\in\sigma} D^2_{(i)} \times \prod_{v_i\notin \sigma} S^1_{(i)} \Big) \subseteq \prod_{v_i\in [m]} D^2_{(i)},
  \end{equation}
  where $D^2_{(i)}$ and $S^1_{(i)}$ are copies indexed by $v_i\in [m]$ of respectively the unit disk
  $D^2=\{z\in \mathbb{C} \mid |z|\leq 1\}$ and the circle $S^1=\partial D^2=\{z\in \mathbb{C} \mid |z|= 1\}$, and $\prod$ denotes Cartesian product of spaces.
  It is shown in~\cite[Section 4]{BP15} that
 $\mathrm{Tor}_{\mathbb{F}[v_1,\cdots, v_m]}(\mathbb{F}[K],\mathbb{F})$ computes the cohomology ring of $\mathcal{Z}_K$, which implies 
  \begin{equation} \label{Equ:DK-ZK}
    \widetilde{D}(K;\mathbb{F}) =\dim_{\mathbb{F}} \mathrm{Tor}_{\mathbb{F}[v_1,\cdots, v_m]}(\mathbb{F}[K],\mathbb{F}) =tb(\mathcal{Z}_K;\mathbb{F}). 
   \end{equation} 
  \n

 \noindent \textbf{Convention:} In the rest of the paper, the coefficients $\mathbb{F}$
will be omitted when there is no ambiguity in the context or the coefficients are not essential for the argument. So we also use
$\widetilde{D}(K)$ and $\widetilde{tb}(K)$ below to denote $\widetilde{D}(K;\mathbb{F})$ and $\widetilde{tb}(K;\mathbb{F})$, respectively.\n

  Next, we prove some easy lemmas on the properties of $\widetilde{D}(K)$.
    
 \begin{lem} \label{Lem-retr}
	If $L$ is a full subcomplex of $K$, then
	$\widetilde{D}(L)\leq \widetilde{D}(K)$.
	\end{lem}

We warn that the statement of Lemma~\ref{Lem-retr} is not true if $L$
 is only a subcomplex  but not a full subcomplex of $K$.\n

 The following lemma characterizes a simplex
 and the boundary of a simplex from the perspective of homology groups.

\begin{lem} \label{Lem:Simplex}
  Let $K$ be a simplicial complex on $[m]$. 
  \begin{itemize}
  \item[(a)] $K\cong \Delta^{[m]}$ if and only if $\widetilde{D}(K)=1$. \n
  \item[(b)]  $K\cong \partial \Delta^{[m]}$ if and only if $\widetilde{H}_{m-2}(K)\neq 0$.
  \end{itemize}
 \end{lem}
 \begin{proof}
 (a) By the formula of $\widetilde{D}(K)$ in ~\eqref{Equ:D-tilde-K-J}, $\widetilde{D}(K)=1$ if and only if all the minimal non-faces of $K$ are empty, which is equivalent to
   $K\cong \Delta^{[m]}$. \n
  
  (b) If $\widetilde{H}_{m-2}(K)\neq 0$, then $K$ is not
  $\Delta^{[m]}$. So $K$ is a subcomplex of $\partial\Delta^{[m]}$. Note that any $(m-2)$-dimensional cycle
  of $\partial\Delta^{[m]}$
  is a multiple of the sum of all the oriented $(m-2)$-simplices of
  $\partial\Delta^{[m]}$. So $K$ must be the whole $\partial\Delta^{[m]}$.
 \end{proof}

\begin{lem} \label{Lem:Join-X}
  For any finite CW-complexes $X$ and $Y$,  
    $$  tb(X\times Y) = tb(X)tb(Y), \ \ \ \widetilde{tb}(X*Y) = \widetilde{tb}(X) \widetilde{tb}(Y)$$
    where $X*Y$ is the join of $X$ and $Y$.
  \end{lem}
  \begin{proof}
   The equality $tb(X\times Y) = tb(X)tb(Y)$ follows from the K\"unneth formula of homology groups.
    In addition, by the homotopy equivalence 
    $$X*Y\simeq \Sigma(X\wedge Y)$$
     where ``$\wedge$'' is the smash product and ``$\Sigma$'' is the reduced suspension, we obtain (with a field coefficient) that
    \begin{equation} \label{Equ:Hom-Iso-join}
      \widetilde{H}_n(X*Y) \cong H_{n-1}(X\wedge Y) \cong
    \bigoplus_{i} (\widetilde{H}_i(X) \otimes \widetilde{H}_{n-1-i}(Y) ). 
    \end{equation}
    The second isomorphism in~\eqref{Equ:Hom-Iso-join}
    follows from the relative version of the K\"unneth formula (see~\cite[Corollary 3.B.7]{Hat02}).
   Notice that $X*Y$ is always path-connected and hence $\widetilde{H}_0(X*Y)=0$. Then it follows that $\widetilde{tb}(X*Y) = \widetilde{tb}(X) \widetilde{tb}(Y)$.
  \end{proof}

 \begin{lem} \label{Lem:Join-K}
  For any finite nonempty simplicial complexes $K$ and $L$,  
  $$ \widetilde{D}(K*L) = \widetilde{D}(K) \widetilde{D}(L) .  $$
  \end{lem}
 \begin{proof}
   Since $\mathcal{Z}_{K*L} \cong\mathcal{Z}_K\times \mathcal{Z}_L$ (see~\cite{BP15}), we obtain from~\eqref{Equ:DK-ZK} and
   Lemma~\ref{Lem:Join-X} that
 \n
 \qquad  $ \widetilde{D}(K*L) = tb(\mathcal{Z}_{K*L}) =tb(\mathcal{Z}_K\times \mathcal{Z}_L)  
   =tb(\mathcal{Z}_K) tb(\mathcal{Z}_L) = \widetilde{D}(K) \widetilde{D}(L)$.
 \end{proof}
 
 \vskip .4cm
 
 \section{Lower bounds of Betti numbers of Stanley-Reisner Ring} \label{Sec:Lower-Bound}
 
 Many studies of the lower bounds of Betti numbers of Stanley-Reisner ring originated from the following long-standing conjecture in commutative algebra
     (see Buchsbaum and Eisenbud~\cite[$\S$1.4]{BuchEisen77} and Hartshorne~\cite[Problem 24]{Hart79}). 
     \n
 \noindent  \textbf{Conjecture} (Buchsbaum-Eisenbud-Horrocks).
 Suppose $R$ is a commutative Noetherian ring such that $\operatorname{Spec}(R)$ is connected, and let $M$ be a nonzero, finitely generated $R$-module of finite projective dimension. Then for any finite projective resolution $0 \rightarrow P_d \rightarrow \cdots \rightarrow P_1 \rightarrow P_0 \rightarrow M \rightarrow 0$ of $M$, we have
$$
\operatorname{rank}_R\left(P_i\right) \geq\binom{ c}{i},
$$
where $c=\operatorname{height}_R\left(\operatorname{ann}_R(M)\right)$, the height of the annihilator ideal of $M$.
\n

The reader is referred to Boocher and Grifo~\cite{BooGri22} for a survey on this conjecture.
 An early result on this conjecture is the following theorem.
 
\begin{thm}[{see~\cite[Corollary 2.5]{EvanGriff88}}] \label{Thm:EvaGriff}
 Let $M$ be a module over the polynomial ring $R=\mathbb{F}[v_1,\cdots,v_m]$ of the form $M= R \slash I $ where $I$ is a monomial ideal. Then 
 \begin{equation}\label{Equ:EvGriff}
  \beta^{-i}(M)\geq \binom{\mathrm{pdim}(M)}{i}, \ \text{for any}\ i\geq 0,
  \end{equation}
 where $\mathrm{pdim}(M)$ denotes the projective dimension of $M$ over $R$.
\end{thm}

 Further studies in this direction include Santoni~\cite{Santo90},  Charalambous, Evans and Miller~\cite{CharEvMill90},  Charalambous and Evans~\cite{CharEvan91}, Charalambous~\cite{Char91},  R\"omer~\cite{Romer02}, Brun and R\"omer~\cite{BrunRomer04}.
\n

  In addition, by Morey and Villarreal~\cite[Corollary 3.33]{MorVilla12}, 
 \begin{equation*} 
   \mathrm{pdim}(\mathbb{F}[K])\geq m-\mdim(K)-1.
   \end{equation*}
Then it follows that 
 \begin{equation*}
  \beta^{-i}(\mathbb{F}[K])\geqslant\binom{m-\mdim(K)-1}{i}.
  \end{equation*}
\n

 In the following theorem, we will give a new proof of this inequality where only the Hochster's formula  and some simple
 Mayer-Vietoris argument are used.
 
 \begin{thm}\label{Thm:Binom-Lower-Bound}
  For any simplicial complex $K$ on $[m]$,
   \begin{equation} \label{Equ:Binom-2-2}
  \beta^{-i}(\mathbb{F}[K]) \geq \binom{m-\mdim(K)-1}{i}, \ \text{for every}\ 0\leq i \leq m-\mdim(K)-1. 
   \end{equation}
   Moreover, for any vertex $v\in V(K)$ and $m_v=|V(\link_K v)|$,  we have for all $i,j\geq 0$
   \n
   
   $\begin{aligned}
    \mathrm{(a)} \   \beta^{-i,2(j+1)}(\mathbb{F}[K]) 
     &\geq   \sum_{0\leq s \leq j} \binom{m-m_v-1}{s}\beta^{-(i-s),2(j-s)}(\mathbb{F}[\link_K v]) \\
     &\qquad \quad \ + \beta^{-i,2(j+1)}(\mathbb{F}[K\setminus v]) - \beta^{-i,2j}(\mathbb{F}[K\setminus v]).
   \end{aligned}$\n
   
    $\begin{aligned}
    \mathrm{(b)}  \   \beta^{-i}(\mathbb{F}[K]) & \geq \sum_{0\leq s \leq i} \binom{m-m_v-1}{s}  \beta^{-(i-s)}(\mathbb{F}[\link_K v]).
   \end{aligned}$\n
   
    $\begin{aligned}
   \mathrm{(c)} \ \wt{D}(K) & \geq 2^{m-m_v-1}\wt{D}(\link_K v).
   \end{aligned}$
   \end{thm} 
  \begin{proof} 
	 For a vertex $v\in V(K)=[m]$ and any subset $J\subset [m]\setminus \{v\}$, consider 
	 $$\left(\star_K v|_{J\cup\{v\}}, (K\setminus v)|_{J\cup\{v\}}\right) = \left(\star_K v|_{J\cup\{v\}}, K|_J \right).$$ 
Obviously
	$$\star_Kv|_{J\cup\{v\}}\cap K|_J=\link_K v|_J,\ \ \star_Kv|_{J\cup\{v\}} \cup K|_J=K|_{J\cup \{v\}}.$$
	So by the Mayer-Vietoris sequence of the pair, we obtain a long exact sequence:
 \begin{equation} \label{Equ:Exact-Seq}
 \cdots \rightarrow \wt{H}_{n+1}(K|_{J\cup\{v\}})\longrightarrow \wt{H}_n(\link_K v|_J) \overset{i_*}{\longrightarrow}\wt{H}_n(K|_J)\longrightarrow \wt{H}_{n}
 (K|_{J\cup\{v\}})\rightarrow
   \end{equation}
	where $i_*$ is induced by the inclusion $i: \link_K v|_J\hookrightarrow K|_J$. 
	It follows that
	\begin{equation}\label{Equ:Reduced-Relation}
		\wt{\beta}_n(\link_K v|_J)\leq \wt{\beta}_{n+1}(K|_{J\cup\{v\}})+\wt{\beta}_n(K|_J), \ \text{for all}\ n\in\Z.
	\end{equation}
	Note that the equality in~\eqref{Equ:Reduced-Relation} holds for all $n\in \Z$ if and only if $i_*$ is surjective for all $n\in \Z$.\n
	In particular, setting $n=|J|-i-1$ and notice $K|_J=(K\setminus v)|_J$, we obtain 
 	\begin{equation}\label{Equ:Reduced-Relation-2}
		\wt{\beta}_{|J|-i-1}(\link_K v|_J) \leq \wt{\beta}_{|J|-i}(K|_{J\cup\{v\}})+\wt{\beta}_{|J|-i-1}((K\setminus v)|_J),\ \text{for any}\ i\in \Z.
	\end{equation}
	\n
	
	If we fix $i$ and sum up~\eqref{Equ:Reduced-Relation-2} for all the subsets $J\subset [m]\setminus \{v\}$ with $|J|=j$, we obtain from the Hochster's formula~\eqref{Equ:Hochster} that \n
	
\begin{itemize}
\item	On the left side of~\eqref{Equ:Reduced-Relation-2}, if $|J\cap \link_K v|=|J|-s$, then
	  $$\wt{\beta}_{|J|-i-1}(\link_K v|_J)=\wt{\beta}_{(|J|-s)-(i-s)-1}(\link_K v|_{J\cap \link_K v}).$$
	  So the sum of the left side of~\eqref{Equ:Reduced-Relation-2} is 
	  $$ \sum_{0\leq s \leq j} \binom{m-m_v-1}{s}\beta^{-(i-s),2(j-s)}(\mathbb{F}[\link_K v]).$$

\item On the right side of~\eqref{Equ:Reduced-Relation-2},
  the sum of the second term gives $\beta^{-i,2j}(\mathbb{F}[K\setminus v])$. But writing the sum of the first term is a little tricky. Indeed, we have
  \begin{align*}
    \beta^{-i,2(j+1)}(\mathbb{F}[K]) 
    &= \underset{|J|=j+1}{\sum_{J\subseteq [m]\backslash \{v\}}} \wt{\beta}_{|J|-i-1}(K|_J) + \underset{|J|=j}{\sum_{J\subseteq [m]\backslash \{v\}}} \wt{\beta}_{|J\cup\{v\}|-i-1}(K|_{J\cup\{v\}}) \\
\big(\text{since}\ K|_J=(K\setminus v)|_J \big)\    &= \beta^{-i,2(j+1)}(\mathbb{F}[K\setminus v]) +
 \underset{|J|=j}{\sum_{J\subseteq [m]\backslash \{v\}}} \wt{\beta}_{|J|-i}(K|_{J\cup\{v\}}).
    \end{align*}
    So the sum of the first term is 
$ \beta^{-i,2(j+1)}(\mathbb{F}[K])- \beta^{-i,2(j+1)}(\mathbb{F}[K\setminus v]) $.
\end{itemize}

\n
 Combining the above argument, we obtain
  	\begin{align*}
	    \sum_{0\leq s \leq j} \binom{m-m_v-1}{s}\beta^{-(i-s),2(j-s)}(\mathbb{F}[\link_K v]) &\leq
	   \beta^{-i,2(j+1)}(\mathbb{F}[K])- \beta^{-i,2(j+1)}(\mathbb{F}[K\setminus v]) \\
	   & + \beta^{-i,2j}(\mathbb{F}[K\setminus v]),
	\end{align*}
 which proves (a). 	Moreover,
	we can easily derive (b)
	 by summing up the inequality in (a) for all $j\geq 0$.\n
	  
	Furthermore, by summing up the inequality in (b) for all $i\geq 0$, we obtain
	  \begin{align*}
	    \wt{D}(K) = \sum_{i\geq 0} \beta^{-i}(\mathbb{F}[K]) & \geq \sum_{i\geq 0} \sum_{0\leq s \leq i} \binom{m-m_v-1}{s}  \beta^{-(i-s)}(\mathbb{F}[\link_K v])   \\
	     & = \sum_{i\geq 0}\sum_{0\leq q\leq i} \binom{m-m_v-1}{i-q}  \beta^{-q}(\mathbb{F}[\link_K v])  \\
	     &=  \sum_{q\geq 0} 2^{m-m_v-1} \beta^{-q}(\mathbb{F}[\link_K v])\\
	     & = 2^{m-m_v-1}\wt{D}(\link_K v).
	  \end{align*}
	This proves (c).\n
	
	To prove the inequality~\eqref{Equ:Binom-2-2}, we do induction on the number of vertices of $K$. Assume that the result is true when $|V(K)|$ is less than $m$.	
	We choose a vertex $v\in V_{\mdim}(K)$.
	Then   
 $\mdim(K)=\mdim(\link_K v)+1 \leq m_v < m$. So by our induction hypothesis,
 \[ \beta^{-(i-s)}(\mathbb{F}[\link_K v]) \geq 
  \binom{m_v-\mdim(\link_K v)-1}{i-s}. \]
 
  Then combining this inequality with (a), we obtain
 \begin{align} 
   \beta^{-i}(\mathbb{F}[K]) &\geq  \sum_{0\leq s \leq i} \binom{m-m_v-1}{s}  \beta^{-(i-s)}(\mathbb{F}[\link_K v])   \label{Equ:Weak-Fk-1} \\   
     &\geq  \sum_{0\leq s \leq i} \binom{m-m_v-1}{s}   
    \binom{m_v-\mdim(\link_K v)-1}{i-s}  \label{Equ:Weak-Fk-2}  \\
    &= \binom{m-\mdim(K)-1}{i}, \notag
   \end{align}
  where the last ``$=$'' is just the Vandermonde identity of binomial coefficients. 
	This finishes the induction and the theorem is proved.
 \end{proof}

 \begin{rem}
 If the $R$-module $M$ in Buchsbaum-Eisenbud-Horrocks conjecture satisfies some extra conditions, there may exist some stronger lower bounds of $\beta^{-i}(M)$. See ~\cite{CharEvMill90, CharEvan91},
   Erman~\cite{Erman10},  Boocher and Seiner~\cite{BooSei18} and Boocher and 
   Wiggleesworth~\cite{BooWig21} for such kind of results.
 \end{rem}

\begin{cor}\label{Cor:bigrad-Formula}
	If $K$ is a taut simplicial complex with $m$ vertices, then for any vertex $v\in V_{\mdim}(K)$,
 $m_v=|V(\link_K v)|$,  we have for all $i,j\geq 0$,\n
 
 $\begin{aligned}
  \mathrm{(a)}   \    \beta^{-i,2(j+1)}(\mathbb{F}[K]) 
     &=   \sum_{0\leq s \leq j} \binom{m-m_v-1}{s}\beta^{-(i-s),2(j-s)}(\mathbb{F}[\link_K v])   \\
     &\qquad + \beta^{-i,2(j+1)}(\mathbb{F}[K\setminus v]) - \beta^{-i,2j}(\mathbb{F}[K\setminus v]).
 \end{aligned}$\n
 
 $\begin{aligned}
 \mathrm{(b)}   \  \beta^{-i}(\mathbb{F}[K])  = \sum_{0\leq s \leq i} \binom{m-m_v-1}{s}  \beta^{-(i-s)}(\mathbb{F}[\link_K v]).
 \end{aligned}$\n

   $\begin{aligned}
 \mathrm{(c)}  \ \wt{D}(K)  = 2^{m-m_v-1}\wt{D}(\link_K v).
 \end{aligned}$\n
 
   $\begin{aligned}
 \mathrm{(d)}  \ \link_K v\ \text{is taut.} 
 \end{aligned}$
 	\end{cor}

	\begin{rem}	  
	 If $K$ is taut but a vertex $u$ does not belong to $V_{\mdim}(K)$, it is possible that
	  $\link_K u$ is not taut 
	  (see Figure~\ref{p:Example-1} for example). 
	\end{rem}

 \begin{figure}[h]
      \includegraphics[width=0.17\textwidth]{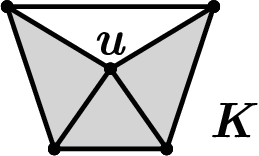}\\
         \caption{$K$ is taut but $\link_K u$ is not ($u\notin V_{\mdim}(K)$).}\label{p:Example-1}
      \end{figure}

 \vskip .4cm
  
  \section{Taut simplicial complexes} \label{Sec-Weakly-Taut}
  In this section, we study various properties of
  taut simplicial complexes.

   \begin{lem} \label{Lem:Weak-Taut-Join}
	Suppose $L$ is a taut simplicial complex.
	Then a simplicial complex $K$ is taut if and only if $K*L$ is taut. In particular,
	$K$ is taut if and only if $K*\Delta^{[r]}$ is taut.
	\end{lem}
	 \begin{proof}
  By Lemma~\ref{Lem:Join-K}, $\widetilde{D}(K*L)=\widetilde{D}(K)\widetilde{D}(L)$.  In addition,
  we have 
  $$ \dim(K*L)=\dim(K)+\dim(L)+1,\  \mdim(K*L)=\mdim(K)+\mdim(L)+1.$$
 Then since $|V(K*L)|=|V(K)|+|V(L)|$,  the lemma follows from the definition of tautness immediately.  
 \end{proof}
	
	 The following theorem provides a useful criterion
	 for us to judge whether a simplicial complex is taut.
	 
  \begin{thm}\label{Thm:Weakly-Taut}
	A simplicial complex	$K$ on $[m]$ is taut if and only if there exists a vertex $v$ such that
		\begin{enumerate}
			\item[(a)] $\link_K v$ is taut,\n
			\item[(b)] for each $J \subset [m]\setminus \{v\}$, the inclusion $\link_K v|_J \hookrightarrow K|_J=(K\setminus v)|_J$ induces a surjection $ \widetilde{H}_*(\link_K v|_J)\rightarrow \widetilde{H}_*((K\setminus v)|_J)$.
		\end{enumerate}
	\end{thm}
	\begin{proof}
	``$\Rightarrow$'' Suppose $K$ is taut. Choose an arbitrary vertex $v\in V_{\mdim}(K)$. Then by Corollary~\ref{Cor:bigrad-Formula} (d),
	 $\link_K v$ is taut. 
	 Moreover, since $K$ is taut, the equality in Theorem~\ref{Thm:Binom-Lower-Bound} (c) holds.
	 This implies that the equality in~\eqref{Equ:Reduced-Relation}
	 holds for all $n\in \Z$. Then in the long exact sequence~\eqref{Equ:Exact-Seq}, the map
	  $i_*: \widetilde{H}_n(\link_K v|_J)\rightarrow \widetilde{H}_n(K|_J) = \widetilde{H}_n((K\setminus v)|_J) $ must be surjective for all $n\in \Z$.\n
	  
	  ``$\Leftarrow$'' The condition (b) implies that
	  for each $J \subset [m]\setminus \{v\}$,
	  the equality in~\eqref{Equ:Reduced-Relation}
	 holds for all $n\in \Z$, and so the equality in~\eqref{Equ:Reduced-Relation-2} holds for all $i\in \Z$. This further implies that 
	  the equality in Theorem~\ref{Thm:Binom-Lower-Bound} (c) hold, i.e.
	   $$\wt{D}(K)=2^{m-m_v-1}\cdot \wt{D}(\link_K v).$$
	    By condition (a), $\wt{D}(\link_K v)=2^{m_v-\mdim(\link_K v)-1}$. So we have
	 \begin{equation} \label{Equ:D-K}
	  \wt{D}(K)=2^{m-1-m_v} \cdot 2^{m_v-\mdim(\link_K v)-1}=2^{m-\mdim(\link_K v)-2}.
	  \end{equation}
	Moreover, since $\widetilde{D}(K) \geq 2^{m-\mathrm{mdim}(K)-1}$ by~\eqref{Equ:uto-mdim}, it follows that 
	 $$\mdim(\link_K v)\leq \mdim(K)-1.$$
	 But obviously $\mdim(\link_K v)\geq \mdim (K)-1$. So $\mdim(\link_K v)= \mdim (K)-1$. Plugging this into~\eqref{Equ:D-K} gives 
	 $\wt{D}(K)= 2^{m-\mathrm{mdim}(K)-1}$, i.e. $K$ is taut.
	\end{proof}
	
	According to the proof of Theorem~\ref{Thm:Weakly-Taut}, we can characterize the vertices in $V_{\mdim}(K)$ when $K$ is taut as follows.
	
	\begin{cor} \label{Cor:mdimv-char}
	 If $K$ is taut, then a vertex $v$ belongs to $V_{\mdim}(K)$ if and only if the conditions in Theorem~\ref{Thm:Weakly-Taut} (a) and (b) hold for $v$.
	 \end{cor}
	  \begin{proof}
	   The ``only if'' direction is already proved in Theorem~\ref{Thm:Weakly-Taut}. If $v$ satisfies the condition in Theorem~\ref{Thm:Weakly-Taut} (b),
	   then by the argument in Theorem~\ref{Thm:Weakly-Taut},
	   we have
	    $$\wt{D}(K)=2^{m-m_v-1}\cdot \wt{D}(\link_K v).$$
	 	   Then since $K$ is taut, i.e. $\wt{D}(K)= 2^{m-\mathrm{mdim}(K)-1}$, we obtain 
	 	   $$\wt{D}(\link_K v)=
	 	   2^{m_v-\mdim(K)}.$$
	Moreover, if $v$ also satisfies the condition (a) in Theorem~\ref{Thm:Weakly-Taut}, then we have
	$\wt{D}(\link_K v)=2^{m_v-\mdim(\link_K v)-1}$. So 
	 	 $\mdim(\link_K v)= \mdim (K)-1$. This implies that
	 	 $v$ is contained in a simplex of dimension $\mathrm{mdim}(K)$, i.e. $v\in V_{\mdim}(K)$.	 	  
	 \end{proof}
	  	
	  \begin{lem} \label{Lem:Disconnected}
	Let $K$ be a simplicial complex on $[m]$ where $m\geq 2$.	
	Then $K$ is taut and disconnected if and only if  $K$ is isomorphic to the disjoint union $\Delta^{[1]}\sqcup\Delta^{[m-1]}$.
	\end{lem}
	\begin{proof}
	 It is easy to check that $\Delta^{[1]}\cup\Delta^{[m-1]}$ is taut by definition. Suppose $K$ is taut and disconnected.
	We first prove that $K$ cannot have more than two connected components. Choose a vertex $v\in V_{\mdim}(K)$
	 and let $L$ the component of $K$ containing $v$. If $K$ has more than two connected components, then there are at least two other components denoted by
  $L'$ and $L''$. Choose a vertex $v'$ of $L'$ and 
  a vertex $v''$ of $L''$, respectively and 
  let $J=\{v',v''\}\subseteq [m]\setminus v$. Then by our assumption,
  $$\link_K v|_J = \link_L v|_J = \{\varnothing\}, \ \ 
     (K\setminus v)|_J = \{v'\}\cup \{v''\}.$$
  But by Theorem~\ref{Thm:Weakly-Taut}, we should have
  a surjection $\widetilde{H}_*(\link_K v|_J)\rightarrow \widetilde{H}_*((K\setminus v)|_J)$, which is a contradiction.\n
    
	So suppose $K$ has only two connected components $L$ and $L'$. Assume that
	$$V(L)=[n]=\{v_1,\cdots, v_n\},\ \ V(L')=[m]\setminus [n]=\{v_{n+1},\cdots, v_m\}.$$ 	
	Then by the formula~\eqref{Equ:D-tilde-K-J} for $\wt{D}(K)$, we obtain
	$$\begin{aligned}
		\wt{D}(K)&=\sum_{J\subset[m]} \wt{tb}(K|_J) \\
		 &=  \wt{tb}(K|_{\varnothing})+ \sum_{\varnothing\neq J\subseteq V(L)} \wt{tb}(K|_J) + \sum_{\varnothing\neq J\subseteq V(L')} \wt{tb}(K|_J) + \underset{J\cap V(L')\neq \varnothing}{\sum_{J\cap V(L)\neq\varnothing}}\wt{tb}(K|_J)	\\
		&\geq 1+ \underset{J\cap V(L')\neq \varnothing}{\sum_{J\cap V(L)\neq\varnothing}}\wt{tb}(K|_J) \\
			&\geq 1+\#\{J\subseteq [m] \,;\, J\cap[n]\neq\varnothing,J\cap[m]\setminus[n]\neq\varnothing\} \qquad\qquad \ \, \\
		&= 1+(2^{n}-1)(2^{m-n}-1) \\
		&= 2^{m-1} + 2 (2^{n-1}-1)(2^{m-n-1}-1) \\
		& \geq 2^{m-1}.
	\end{aligned}$$
	
	On the other hand, $\wt{D}(K)=2^{m-\mdim(K)-1}\leq 2^{m-1}$. Hence $\wt{D}(K)=2^{m-1}$ and all the ``$\geq$'' above must be ``$=$''. This implies:
$$\begin{aligned}
	\bullet  \ &\sum_{\varnothing\neq J\subseteq V(L)} \wt{tb}(K|_J) = \sum_{\varnothing\neq J\subseteq V(L')} \wt{tb}(K|_J)=0,\ \text{and so by Lemma~\ref{Lem:Simplex} both
	}\\
	& \text{$L$ and $L'$ are simplices}.\\
	\bullet \ &\text{Either $2^{n-1}-1=0$ or $2^{m-n-1}-1=0$, i.e. either $n=1$ or $m-n=1$}.
	\end{aligned}
	$$
	Therefore, $K$ must be isomorphic to $\Delta^{[1]}\sqcup\Delta^{[m-1]}$.
	\end{proof}
	
	We immediately have the following corollary from Lemma~\ref{Lem:Disconnected}.
	
	\begin{cor} \label{Cor:Disconnected}
	 If a simplicial complex $K$ on $[m]$ is taut where $m\geq 2$, then the following statements are all equivalent:
	 \begin{itemize}
	  \item[(a)] $K$ is disconnected.\n
	  \item[(b)] $\mdim(K)=0$.\n
	  \item[(c)] There exists a vertex $v\in V_{\mdim}(K)$ such that $\link_K(v)=\{\varnothing\}$.\n
	  \item[(d)] $K$ is isomorphic to $\Delta^{[1]}\sqcup\Delta^{[m-1]}$.
	  \end{itemize}
	\end{cor}
	
	\n

	\begin{thm}\label{Thm:Subcomplex-Taut}
			Every full subcomplex of a taut simplicial complex is taut.
		\end{thm}
	\begin{proof}
	 Let $K$ be a taut simplicial complex
	 on $[m]$. The statement is clearly true when $K$ is disconnected (by Lemma~\ref{Lem:Disconnected}) or when $\dim(K)=0$.
	 So we assume that $K$ is connected and $\dim(K)\geq 1$ in the following.\n
	  By induction on the number of vertices of $K$, we only need to prove that $K\setminus v$ is taut for every $v\in V(K)$. Note that since $K$ is connected, 
 we must have $\mdim(K)\geq 1$ by Corollary~\ref{Cor:Disconnected}. So we can take
 a vertex $w\in V(K)\setminus \{v\}$ with $w\in V_{\mdim}(K)$.  By Corollary~\ref{Cor:bigrad-Formula} (d), $\link_K w$ is taut. Then since $\link_K w$ has less vertices
 than $K$,  by our induction its full subcomplex $(\link_K w)\setminus v$ is also taut. In the following, we use Theorem~\ref{Thm:Weakly-Taut} to prove that $K\setminus v$ is taut.
  \begin{itemize}
    \item[(i)] First, by definition $\link_{K\setminus v}w = (\link_K w)\setminus v$. So $\link_{K\setminus v}w$ is taut.
   \n
    \item[(ii)] Second, by Theorem~\ref{Thm:Weakly-Taut} (b)  
 the inclusion $\link_K w|_J \hookrightarrow (K\setminus w)|_J$ induces a surjection $ \widetilde{H}_*(\link_K w|_J)\rightarrow \widetilde{H}_*((K\setminus w)|_J)$ for each $J \subset [m]\setminus w$.
 So in particular for every $J'\subseteq [m]\setminus\{v,w\}=V(K\setminus v)\setminus \{w\}$,
 the inclusion 
 $$\link_{K\setminus v} w|_{J'} = \link_{K} w|_{J'} \hookrightarrow  (K\setminus w)|_{J'}=((K\setminus v)\setminus w)|_{J'}$$
  induces a surjection $ \widetilde{H}_*(\link_{K\setminus v} w|_{J'})\rightarrow \widetilde{H}_*(((K\setminus v)\setminus w)|_{J'})$.
 \end{itemize}
 \n
 So we can assert that $K\setminus v$ is taut
 by applying Theorem~\ref{Thm:Weakly-Taut} to 
 the vertex $w\in V(K\setminus v)$. The theorem is proved. 
\end{proof}
	 
	\n
	
\begin{rem}
  It is possible that a simplicial complex $K$ is not taut while all its proper full subcomplexes are taut, see Figure~\ref{p:Example-2} for example.
\end{rem}	

 \begin{figure}[h]
  \includegraphics[width=0.23\textwidth]{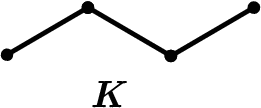}\\
          \caption{All proper full subcomplexes of $K$ are taut, but $K$ is not.} \label{p:Example-2}
      \end{figure}
 
  \begin{cor} \label{Cor:two-vertex}
	Let $K$ be a taut simplicial complex. For any vertex $v$ of $K$, if $w \in V_{\mathrm{mdim}}(K) \backslash v$, then $ w \in V_{\mathrm{mdim}}(K \backslash v)$.
	 \end{cor}
	 \begin{proof}
	  By our assumption, we can deduce the conclusions (i) and (ii) in middle of the proof of Theorem~\ref{Thm:Subcomplex-Taut}.
	  Then by applying
	   Corollary~\ref{Cor:mdimv-char} to the vertex $w$ and
	   $K\backslash v$, we can assert that $ w \in V_{\mathrm{mdim}}(K \backslash v)$.
	 \end{proof}
	 
The following theorem tells us more about the structure
of a taut simplicial complex $K$ at a vertex
$v\in V_{\mdim}(K)$.

	\begin{thm}\label{Thm:K-v-Join}
		Suppose $K$ is a taut simplicial complex.
		For a vertex $v\in V_{\mdim}(K)$, let $r_v= |V(K)| -  |V(\star_Kv)|$. Then 
		$$K\setminus v \cong K|_{V(\link_Kv)}*\Delta^{[r_v]}.$$
		\end{thm}
\begin{proof}
 \textbf{Claim-1:} For each $J\subset V(K)\setminus V(\star_K v)$, $(K\setminus v)|_J$ is a simplex. \n
 
 We prove this by induction on the cardinality $|J|$.
 When $|J|=1$, $(K\setminus v)|_J$ is a vertex. Assume that the claim is true when $|J|\leq k-1$. Take an arbitrary $J\subset V(K)\setminus V(\star_K v)$ with $|J|=k$. By our induction hypothesis, $(K\setminus v)|_{J'}$ is a simplex for any $J'\subsetneq J$. So if the claim is not true for $J$, we would have
  $(K\setminus v)|_J\cong\partial\Delta^{[k]}$ and $\link_K v|_J =\{\varnothing\}$. But this contradicts the surjectivity of the map $\wt{H}_*(\link_Kv|_J)\rightarrow \wt{H}_*((K\setminus v)|_J)$ (see Theorem~\ref{Thm:Weakly-Taut} (b)). This proves Claim-1.
 \n
 
 By Claim-1, $K|_{V(K)\setminus V(\star_K v)}$ is a simplex.
 For brevity, we denoted it by $\Delta^{[r_v]}$.\n
 
 \textbf{Claim-2:} For every simplex $\sigma$ of $K|_{V(\link_Kv)}$, $\sigma * \Delta^{[r_v]}$ is a simplex of $K\setminus v$. \n
 
 Note that $\link_K v$ may not be a full subcomplex of $K$
 (see Figure~\ref{p:Example-1} for example). It is well possible that $K|_{V(\link_Kv)}$ has more simplices than $\link_K v$. \n
 
 We prove this claim by induction on the cardinality $|\sigma|$ of $\sigma$. 
 If $|\sigma|=0$, this is true since $\Delta^{[r_v]}$ is a simplex of $K\setminus v$. Assume that the statement holds when $|\sigma|<k$. Let $\sigma$ be a simplex of $K|_{V(\link_Kv)}$ with $|\sigma|=k\geq 1$. 
 If $\sigma*\Delta^{[r_v]}$ is not a simplex of $K\setminus v$, let 
 $J=(V(K)\setminus V(\star_K v))\cup V(\sigma)$. Then
  by our induction hypothesis,\n
 \begin{itemize}
 \item $(K\setminus v)|_{J}=\sigma\cup (\partial\sigma*\Delta^{[r_v]})$
  which is homotopy equivalent to $\partial \Delta^{[k+1]}$.
  \n
  \item $(K\setminus v)|_{V(\sigma)} = \sigma \cong \Delta^{[k]}$.\n
  
  \item $\link_K v|_J = \link_K v |_{V(\sigma)}  \subseteq (K\setminus v)|_{V(\sigma)} \subseteq (K\setminus v)|_J $.\n
  \end{itemize}
  Then the following composition of maps 
 $$\wt{H}_*(\link_K v|_J)\rightarrow \wt{H}_*\big((K\setminus v)|_{V(\sigma)} \big)\rightarrow \wt{H}_*((K\setminus v)|_J)\cong \wt{H}_*(\partial \Delta^{[k+1]}) $$ 
 cannot be surjective since the second map is not. But this contradicts our assumption $v\in V_{\mdim}(K)$ because of Theorem~\ref{Thm:Weakly-Taut} (b). So $\sigma*\Delta^{[r_v]}$ has to be a simplex of $K\setminus v$, which proves Claim-2. 
  \n
 
 It follows immediately from Claim-2 that $K\setminus v\cong K|_{V(\link_Kv)}*\Delta^{[r_v]}$. 
 \end{proof}

 According to Theorem~\ref{Thm:K-v-Join}, we can decompose a taut simplicial complex $K$ at a vertex $v\in V_{\mdim}(K)$ as:
 \begin{align} \label{Equ:Decomposition}
	  K &=\star_K v\cup_{\link_K v} (K\setminus v) \\
	    &\cong (v*\link_K v)\cup_{\link_K v}
	     \big(K|_{V(\link_Kv)}*\Delta^{[r_v]}\big). \notag
	 \end{align}
	 A schematic  picture of this decomposition is shown in Figure~\ref{p:Global-Picture}. Moreover, by Corollary~\ref{Cor:bigrad-Formula} (d), $\link_K v$ is taut. By Theorem~\ref{Thm:Subcomplex-Taut}, $K\setminus v$ is taut, which implies that $K|_{V(\link_K v)}$ is also taut (by Lemma~\ref{Lem:Weak-Taut-Join}). 
	 	Note that when $\link_K v =\{\varnothing\}$, $K$ is the disjoint union $\{v\}\sqcup \Delta^{[r_v]}$ (see Corollary~\ref{Cor:Disconnected}). 
	
 \begin{figure}[h]
         \includegraphics[width=0.53\textwidth]{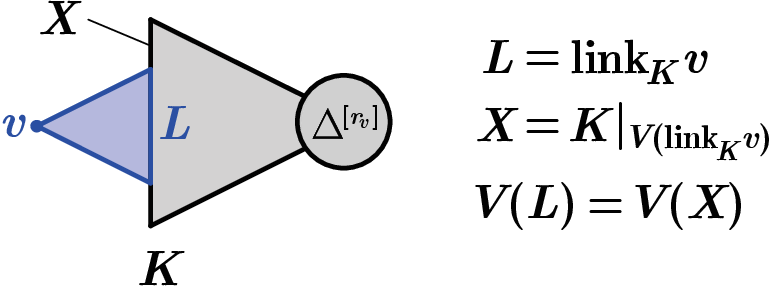}\\
          \caption{A schematic  picture of a taut simplicial complex $K$ at a vertex $v\in V_{\mdim}(K)$}\label{p:Global-Picture}
      \end{figure}

\begin{rem}
  In Figure~\ref{p:Global-Picture}, if $L=X$, then
  $K = \begin{cases}
    \text{Cone}(X),  &  \text{if $r=0$}; \\
     \text{Susp}(X),  &  \text{\text{if $r=1$}}.
 \end{cases}$

 Here $\text{Susp}(X)$ denotes the \emph{suspension} of $X$. So the decomposition of $K$ in~\eqref{Equ:Decomposition} can be thought of as a generalization of taking cone or suspension of a space.
\end{rem}

   As we will see below, the decomposition~\eqref{Equ:Decomposition} is very useful for us to study the global structure of a taut simplicial complex.    
Next, we prove some properties of $\dim(K)$ and 
$\mdim(K)$ for a taut simplicial complex $K$.

\begin{lem}\label{Lem:mdim-link-v}
  Suppose $K$ is a taut simplicial complex.
	Then for every vertex $v\in V_{\mdim}(K)$, we have\n
	\begin{itemize}
	\item[(a)] $\mdim(\link_K v) \leq \mdim(K\setminus v)$, and\n
	
	\item[(b)] $\mdim(\link_K v) = \mdim(K\setminus v)$ if and only if $\link_K v=K\setminus v$.
	\end{itemize}   	
 \end{lem}
 \begin{proof}
 (a) Suppose $V(K)=[m]$. First, we have $\mdim(\link_K v) = \mdim(K)-1$ since $v\in V_{\mdim}(K)$. By Corollary~\ref{Cor:bigrad-Formula} (d) and Theorem~\ref{Lem:Disconnected}, $\link_K v$ and $K\setminus v$ are both taut. So we have
 \begin{equation} \label{Equ:Y-L}
     \widetilde{D}(\link_K v)= 2^{m-\mdim(\link_K v)-1}, \ \ \widetilde{D}(K\setminus v)= 2^{m-\mdim(K\setminus v)-1}.
   \end{equation}
 Moreover, from Theorem~\ref{Thm:Weakly-Taut} (b) and the Hochster's formula, we can deduce that $\widetilde{D}(\link_K v)\geq \widetilde{D}(K\setminus v)$. Then by~\eqref{Equ:Y-L}, $ \mdim(\link_K v) \leq \mdim(K\setminus v)$.
  \n 
 (b) By Theorem~\ref{Thm:K-v-Join}, we can write $K$ as 
  $$K=(\star_K v)\cup_{\link_K v} (K\setminus v) = (v*\link_K v)\cup_{\link_K v} \big(K|_{V(\link_K v)}*\Delta^{[r_v]}\big).$$
  
  If $r_v\geq 1$, we have $\Delta^{max}(K) = \Delta^{max}(v*\link_K v) \cup \Delta^{max}(K\setminus v)$. Then by the definition of $\mdim(K)$, we obtain $\mdim(K)\leq \mdim(K\setminus v)$ and so
	$$\mdim(\link_K v) \leq \mdim(K\setminus v)-1.$$
	Therefore, if $\mdim(\link_K v)=\mdim(K\setminus v)$, it is necessary
	that $r_v=0$ and hence $V(K\setminus v)=V(\link_K v)$. In this case, we have $\widetilde{D}(\link_K v)= \widetilde{D}(K\setminus v)$ by~\eqref{Equ:Y-L}. So by Theorem~\ref{Thm:Weakly-Taut} (b),  the inclusion $i:\link_K v \hookrightarrow K\setminus v$ must induce an isomorphism $i_*:\widetilde{H}_*(\link_K v|_J)\rightarrow \widetilde{H}_*((K\setminus v)|_J)$ for every $J\subseteq [m]$.
 If $\link_K v\neq K\setminus v$, we choose a minimal (by inclusion)
  simplex $\sigma$ in  $(K\setminus v)\setminus \link_K v$. Since $V(K\setminus v)=V(\link_K v)$,
 $\dim(\sigma)\geq 1$. Then
  $\link_K v|_{V(\sigma)}=\partial \sigma$ is a simplicial sphere while $(K\setminus v)|_{V(\sigma)}=\sigma$. This contradicts that
  $i_*:\widetilde{H}_*(\link_K v|_{V(\sigma)})\rightarrow \widetilde{H}_*((K\setminus v)|_{V(\sigma)})$ is an isomorphism.
 Therefore, we must have $\link_K v=K\setminus v$ if $\mdim(\link_K v)=\mdim(K\setminus v)$. 
 \end{proof}
 \n

\begin{lem}\label{Lem:Decomp-Quotient}
For a taut simplicial complex $K$ and a vertex $v\in V_{\mdim}(K)$,
	let $X=K|_{V(\link_K v)}$ and $L=\link_K v$.  Then  we have either $\|K\| \simeq  \|X\| \slash \|L\|$ or $\|K\| \simeq 
	\mathrm{Susp}(\|L\|)$. 
\end{lem}
\begin{proof}
 By~\eqref{Equ:Decomposition},  $K = (v*L)\cup_L (X*\Delta^{[r_v]})$.
	  Then we obtain
	 \begin{align*}
	 \|K\| &= (v*\|L\|) \cup_{\|L\|} \big(\|X*\Delta^{[r_v]}\|\big) \\
	 &\simeq \begin{cases}
   (v*\|L\|) \cup_{\|L\|} \big(v'*\|X\|\big)\simeq
   \text{Susp}(\|L\|),  &  \text{if $r_v>0$}; \\
  (v*\|L\|) \cup_{\|L\|} \|X\| \simeq \|X\|\slash \|L\|,  &  \text{if $r_v=0$}.
 \end{cases}  
	 \end{align*}
\end{proof}

\begin{thm}\label{Thm:Homotopy-Type}
 Let $K$ be a taut simplicial complex over a field $\mathbb{F}$. Then its geometric realization $\|K\|$ is homotopy equivalent either to a point or to a sphere of dimension $\mdim(K)$.
\end{thm}
\begin{proof}
  We do induction on the number of vertices.
  The one vertex case is trivial.  Suppose $K$ is a simplicial complex on $[m]$ and we choose a vertex $v\in V_{\mdim}(K)$. Then
  by~\eqref{Equ:Decomposition}, we can write $K = (v*L)\cup_L (X*\Delta^{[r_v]})$ where  
 $L=\link_K v$ and $X=K|_{V(\link_K v)}$ are both taut subcomplexes of $K$ with $|V(X)|<m$ and $|V(L)|<m$.
 So by induction, both $\|X\|$ and $\|L\|$  are 
  homotopy equivalent either to a point or to a sphere
  of dimension $\mdim(X)$ and $\mdim(L)$, respectively.
   In addition,  since $v\in V_{\mdim}(K)$, we have $\mdim(K)=\mdim(L)+1$. Moreover,
  according to Theorem~\ref{Thm:Weakly-Taut} (b), the inclusion
 $i: L\hookrightarrow K\setminus v$ induces a surjection  $i_*: H_*(L|_J;\mathbb{F})\rightarrow H_*((K\setminus v)|_J;\mathbb{F})$ for any $J\subseteq [m]\setminus v$. In particular, by setting $J=V(\link_K v)$, we have a surjection
  \begin{equation} \label{Equ:Surj-L-X}
   i_*: H_*(L;\mathbb{F})\rightarrow H_*(X;\mathbb{F}).
   \end{equation} 
  
  By Lemma~\ref{Lem:Decomp-Quotient}, we have the following cases:
  \n
  \begin{itemize}
   \item[\text{Case 1:}] $\|K\| \simeq   \text{Susp}(\|L\|) $.
   Then we have
      $\|K\| \simeq \begin{cases}
    pt,  &  \text{if $\|L\|\simeq pt$}; \\
    S^{\mdim(L)+1},  &  \text{if $\|L\|\simeq S^{\mdim(L)}$}.
 \end{cases}$
 
  \n
 
   \item[\text{Case 2:}] $\|K\| \simeq  \|X\|\slash \|L\|  $ where $X=K\setminus v$.
  \n
          \begin{itemize}   
   \item[(i)] If $\|X\|\simeq pt$, 
    $\|K\| \simeq \text{Susp} (\|L\|) \simeq \begin{cases}
    pt,  &  \text{if $\|L\|\simeq pt$}; \\
    S^{\mdim(L)+1} ,  &  \text{if $\|L\|\simeq S^{\mdim(L)}$}.
 \end{cases} 
  $\n 
  \item[(ii)] If $\|X\|\simeq S^{\mdim(X)}$, then since $i_*: H_*(L;\mathbb{F})\rightarrow H_*(X;\mathbb{F})$ is surjective,
   $\|L\|$ cannot be homotopy equivalent to a point. So 
   $\|L\|\simeq S^{\mdim(L)}$ by our induction, and we must have $\mdim(L)=\mdim(X)$. Moreover, since $X=K\setminus v$  in this case, it follows from Lemma~\ref{Lem:mdim-link-v} (b) that
    $L=X$. Therefore, $\|K\| \simeq  \|X\|\slash \|L\| =pt$.
    
 \end{itemize}   
 
 \end{itemize}
 
 \n
       
  By the above argument, $\| K \|$ is either contractible 
  or homotopy equivalent to $S^{\mdim(L)+1}=S^{\mdim(K)}$.
  So we finish the induction.
\end{proof}

\begin{cor}\label{Thm:Equiv} 
  If a simplicial complex $K$ is taut over a field $\mathbb{F}$, then
  \begin{itemize}
  \item[(a)] Every full subcomplex of $K$ is
  homotopy equivalent either to a point or to a sphere;\n
  
  \item[(b)] $K$ is taut over an arbitrary field. 
  \end{itemize}
\end{cor}
\begin{proof}
(a)  This follows immediately from Theorem~\ref{Thm:Subcomplex-Taut} and Theorem~\ref{Thm:Homotopy-Type}.\n
  
  (b) Let $\mathbb{G}$ denote an arbitrary field. For every $J\subseteq V(K)=[m]$, the full subcomplex $K|_J$ satisfies
  $\widetilde{tb}(K|_J;\mathbb{G}) =
 \widetilde{tb}(K|_J;\mathbb{F})$ 
   since the geometric realization of $K|_J$ is homotopy equivalent either to a point or to a sphere. Then
  $$ \widetilde{D}(K;\mathbb{G}) = \sum_{J\subseteq [m]}
 \widetilde{tb}(K|_J;\mathbb{G}) = \sum_{J\subseteq [m]}
 \widetilde{tb}(K|_J;\mathbb{F}) =  \widetilde{D}(K;\mathbb{F})=2^{m-\mdim(K)-1}. $$
 So $K$ is taut over $\mathbb{G}$ as well.
  \end{proof}
 
 By the Hochster's formula for the bigraded Betti numbers (see Theorem~\ref{Thm:Hoch} and~\eqref{Equ:Hochster}), we obtain the following topological characterization of taut simplicial complexes  immediately from Corollary~\ref{Thm:Equiv}.

\begin{cor}\label{Cor:Character}
 Let $K$ be a simplicial complex with $m$ vertices. The following statements are equivalent:
 \begin{itemize}
 \item[(a)] $K$ is taut. \n
 \item[(b)] Each full subcomplex of $K$ is homotopy equivalent either to a point or to a sphere, and the number of full subcomplexes of $K$ that are homotopy equivalent to a sphere is $2^{m-\mdim(K)-1}$. 
 \end{itemize}
\end{cor}

 In addition, tautness puts some restriction on the dimension of a simplicial complex
$K$ relative to the vertex number of $K$. \n

 	\begin{cor} \label{Cor:Weak-Taut-Range}
		If a simplicial complex $K$ with $m$ vertices is taut, then 
			 $$\left[\frac{m-1}{2}\right] \leq \dim(K) \leq m-1.$$
\end{cor}
	\begin{proof}
	 We do induction on the number of vertices of $K$.
	 Assume that the result is true when $|V(K)| < m$. Take a vertex $v\in V_{\mdim}(K)$. Then by Theorem~\ref{Thm:K-v-Join},
			$$K\setminus v\cong K|_{V(\link_Kv)}*\Delta^{[r_v]},\ \text{where}\ r_v= |V(K)| -  |V(\star_Kv)|.$$
			Note that $|V(\link_K v)| = m-r_v-1$,
	 	Moreover, by Theorem~\ref{Thm:Weakly-Taut} and Theorem~\ref{Thm:Subcomplex-Taut}, both $\link_Kv$ and $K|_{V(\link_Kv)}$ are taut. So by the induction hypothesis,
	 	$$\ \ \quad \dim(\link_K v) \geq \left[ \frac{m-r_v-1-1}{2} \right] = \left[ \frac{m-r_v}{2} \right]-1,$$
	 	$$\dim(K|_{V(\link_Kv)})\geq \left[\frac{m-r_v-1-1}{2}\right] = \left[\frac{m-r_v}{2}\right]-1.$$
	 	
	 	Hence
		\begin{align*}
		  \dim (K) & \geq \max \{\dim(\star_K v),\dim (K\setminus v)\}\\
		  &\geq \max \left\{\left[\frac{m-r_v}{2}\right],\left[\frac{m+r_v}{2}\right]-1\right\} 
		  = \begin{cases}
   \left[ \frac{m}{2} \right] ,  &  \text{if $r_v=0$  }; \\
   \left[ \frac{m-1}{2} \right]  ,  &  \text{if $r_v=1$};\\
   \geq \left[ \frac{m}{2} \right] ,  &  \text{if $r_v\geq 2$}.   
 \end{cases}
		  \end{align*}
		  So we always have $\dim(K)\geq \left[\frac{m-1}{2}\right]$. 
		  \end{proof}
		  
		  \begin{rem} The lower bound of $\dim(K)$ in 
		  Corollary~\ref{Cor:Weak-Taut-Range} can be reached by the taut simplicial complexes $\partial \Delta^{[2]}*\partial \Delta^{[2]} *\cdots*\partial \Delta^{[2]}$ when $m$ is even and by $\Delta^{[1]}*\partial \Delta^{[2]}*\partial \Delta^{[2]} *\cdots*\partial \Delta^{[2]}$ when $m$ is odd.
\end{rem}
		
		\vskip .4cm

		\section{Simplicial complexes with minimal Taylor resolution} \label{Sec:Taylor}

	Let $K$ be a simplicial complex with vertex set $V(K)=\{v_1,\cdots, v_m\}$.
	For a subset $N\subset V(K)$, let $v_N \in \mathbb{F}\left[v_1, \ldots, v_m\right]$ 
	denote the square-free monomial corresponding to
 $N$, and we set $v_{\emptyset}=1$.
	\n
	
		 Let $N_1, \ldots, N_r \subseteq V(K)$ be all the minimal non-faces of $K$. Then we have
$$
\mathbb{F}[K]=\mathbb{F}\left[v_1, \ldots, v_m\right] /\left(v_{N_1}, \ldots, v_{N_r}\right) .
$$
The \emph{Taylor resolution} for $\mathbb{F}[K]$ is the free $\mathbb{F}\left[v_1, \ldots, v_m\right]$-module resolution
$$
\cdots \overset{d}{\longrightarrow} R^{-\ell} \overset{d}{\longrightarrow} R^{-\ell+1} \overset{d}{\longrightarrow} \cdots \overset{d}{\longrightarrow} R^0=\mathbb{F}\left[v_1, \ldots, v_m\right] \longrightarrow \mathbb{F}[K],
$$
such that $R^{-\ell}$ is the free $\mathbb{F}\left[v_1, \ldots, v_m\right]$-module generated by symbols $w_{i_1, \ldots, i_{\ell}}$ for $1 \leqslant i_1<\cdots<i_{\ell} \leqslant r$ with the differential
$$
d\left(w_{i_1, \ldots, i_{\ell}}\right)=\sum_{k=1}^{\ell}(-1)^{k+1} v_{N_{i_k}-(N_{i_1} \cup \ldots \cup \widehat{N_{i_k}} \cup \ldots \cup N_{i_{\ell}})} w_{i_1, \ldots, \widehat{i_k}, \ldots, i_{\ell}},
$$
 We say that the Taylor resolution is \emph{minimal} if the differential satisfies
$$
d \otimes_{\mathbb{F}\left[v_1, \ldots, v_m\right]} \mathbb{F}=0.
$$

 We say that $K$ \emph{has a minimal Taylor resolution} if the Taylor resolution of $\mathbb{F}[K]$ is minimal for some $\mathbb{F}$. In fact, this notion does not depend on $\mathbb{F}$. 
  Moreover, it was shown in~\cite{Anton16} and~\cite{IriKishi18} that the minimality of the Taylor resolution of a simplicial complex can be interpreted combinatorially as follows.
  Let 
  $$[m]=\{v_1,\cdots, v_m\}, \ \ \mathbf{a}=\left\{v_{1}, \cdots, v_{r}\right\}, \ 1\leq r \leq m.$$ 
  
  Suppose 
  $\mathbf{N}=\left\{N_1, \cdots, N_r\right\}$ is a collection of subsets of $[m]\backslash \mathbf{a}=\{v_{r+1},\cdots, v_m\}$, where 
   it is possible that $N_i=N_j$ for some $i\neq j$. 
   Moreover, define
     $$ \widetilde{N}_{i}=N_{i} \cup \{ v_{i}\}, \ 1\leq i \leq r. $$ 
  
 \begin{defi} \label{Defi:KN}
  Let $K(\mathbf{N})$ denote the simplicial complex with vertex set $[m]$ whose minimal non-faces are $\widetilde{N}_1, \cdots, \widetilde{N}_r$. If $\mathbf{N}$ is empty, define $K(\mathbf{N})= \Delta^{[m]}$. 
  \end{defi}

 \begin{thm}[{\cite[Lemma 4.11]{Anton16}, \cite[Proposition 2.1]{IriKishi18}}] \label{Thm:IriKishi}
  A simplicial complex $K$ on $[m]$ has a minimal Taylor resolution if and only if there is a sequence $\mathbf{N}$
   such that $K \cong K(\mathbf{N})$.
 \end{thm}

Next, we prove that a simplicial complex $K$ is taut if and only if $K$ has a minimal Taylor resolution.
 In the following, let 
  $$\mu(K)= \text{the
number of minimal non-faces of $K$}.$$

\begin{lem}\label{Lem:K-N-1}
For any sequence $\mathbf{N}=\{N_1,\cdots, N_r\}$ of subsets of $ [m] \backslash \mathbf{a}$ where $\mathbf{a}=\left\{v_{1}, \cdots, v_{r}\right\}$,
 $\mu(K(\mathbf{N}))=m-\mathrm{mdim}(K(\mathbf{N}))-1$.
 \end{lem}
\begin{proof}
 Since $\mu(K(\mathbf{N}))=r$, it is equivalent to prove
that $\mathrm{mdim}( K(\mathbf{N}))=m-r-1$. 
Note that any maximal face $\xi$ of $K(\mathbf{N})$ as a set can be written as a disjoint union $\xi=\xi_{\mathbf{a}} \cup \xi_{\mathbf{a}^{c}}$ where $\xi_{\mathbf{a}} \subset \mathbf{a}$ and  $\xi_{\mathbf{a}^c} \subset \mathbf{a}^c=[m] \backslash \mathbf{a}$.\n
\begin{itemize}
\item If $\xi_{\mathbf{a}}=\varnothing$, then by the definition of $K(\mathbf{N})$, we must have $\xi=\mathbf{a}^c$ since $\xi$ is maximal. So $\mathrm{dim}(\xi)=m-r-1$.\n

\item If $\xi_{\mathbf{a}} \neq \varnothing$, say $\xi_{\mathbf{a}}=\left\{v_1, \cdots, v_s\right\}$ for some $s \leqslant r$. Since a subset of $[m]$ is face of $K(\mathbf{N})$ if and only if it contains no minimal non-faces and since $\xi$ is maximal, $\xi_{\mathbf{a}^{c}}$ must be the form $\mathbf{a}^{c} \backslash \cup^s_{\ell=1} v_{i_{\ell}} $ with
 $v_{i_{\ell}} \in N_{\ell}$ for each $1 \leqslant \ell \leqslant s$. Note that $N_{\ell}$ and $N_{\ell'}$ may have common vertices. So
$$
\mathrm{dim} (\xi) =\left| \xi_{\mathbf{a}}\right|+\left|\xi_{\mathbf{a}^c} \right|-1 \geqslant s+(m-r)-s-1=m-r-1 .
$$
\end{itemize}
So we can conclude that $\mathrm{mdim}(K(\mathbf{N}))= m-r-1$.
\end{proof}

\begin{lem} \label{Lem:K-N-2}
 For any simplicial complex $K$ with $m$ vertices, 
 $$\mu(K)\geq m-\mathrm{mdim}(K)-1.$$
 Moreover, if $\mu(K) = m-\mathrm{mdim}(K)-1$, then
 $K$ has a minimal Taylor resolution.
\end{lem}
\begin{proof}
 Fix a maximal simplex $\xi$ of $K$ with $\mathrm{dim}(\xi)=\mathrm{mdim}(K)$ and take an arbitrary vertex $v \in V(K) \backslash V(\xi)$. Since $\xi \cup \{v\}$ is not a face of $K$, it contains a minimal non-face of the form $\xi_v \cup \{v\}$ where $\xi_v \subseteq \xi$. Moreover, for distinct vertices $v, w \in V(K) \backslash V(\xi)$, the minimal non-faces $\xi_v \cup \{v\}$ and $\xi_w \cup \{w\}$ are clearly different. Hence the number of minimal non-faces of $K$ is no less than $m-\mathrm{dim}(\xi)-1=m-\mathrm{mdim} (K)-1$.\n
 Moreover, when $\mu(K) = m-\mathrm{mdim}(K)-1$,
 all the minimal non-faces of $K$ are
 $$ \big\{ \xi_v \cup \{v\} \mid \xi_v\subseteq \xi,\ \forall\, v \in V(K)\backslash V(\xi) \big\}.$$
 Then $K\cong K(\mathbf{N})$ where $\mathbf{a}= V(K)\backslash V(\xi)$ and $\mathbf{N}=\{ \xi_v \mid \forall\, v\in V(K)\backslash V(\xi) \}$. So by Theorem~\ref{Thm:IriKishi}, $K$ has a minimal Taylor resolution.
\end{proof}

 The following is the main theorem of this section which 
 relates tautness of a simplicial complex $K$ to
 the minimality of the Taylor resolution of $\mathbb{F}[K]$.
  
 \begin{thm} \label{Thm:Taylor-2}
 Let $K$ be a simplicial complex on $[m]$. The following conditions are equivalent:
 \begin{itemize}
 \item[(a)] $K$ is taut.\n
 \item[(b)] The number of minimal non-faces of $K$ is equal to $m-\operatorname{mdim}(K)-1$.\n
 \item[(c)] $K$ has a minimal Taylor resolution.
 \end{itemize} 
 \end{thm}
 \begin{proof}
    ``(b) $\Rightarrow$ (c)'' This is contained in Lemma~\ref{Lem:K-N-2}.\n
    
    ``(c) $\Rightarrow$ (a)'' 
      Since the Taylor resolution of $\mathbb{F}[K]$ is minimal, we have
    $$\beta^{-i}(\mathbb{F}[K]) = \binom{\mu(K)}{i}, \ \text{for every}\ 0\leq i \leq \mu(K).$$
    By Theorem~\ref{Thm:IriKishi}, $K\cong K(\mathbf{N})$ for some sequence $\mathbf{N}$ of subsets of $V$. Then by Lemma~\ref{Lem:K-N-1}, $\mu(K)= m-\mathrm{mdim}(K)-1$.
    This implies that $K$ is taut.\n
    
    ``(a) $\Rightarrow$ (b)'' We do induction on $m$. The result is clearly true when $m=1$.
    Choose an arbitrary vertex $v \in V_{\mathrm{mdim}}(K)$.  Since $K$ is taut, so is $\link_K v$ and we have $\mathrm{mdim}(\link_K v) = \mathrm{mdim}(K)-1$.
    Moreover, by Corollary~\ref{Cor:mdimv-char} the condition Theorem~\ref{Thm:Weakly-Taut} (b) holds for
    the vertex $v$. This implies that for every subset $J \subset[m] \backslash \{v\}$, the equality in~\eqref{Equ:Reduced-Relation} holds:
    \begin{equation} \label{Equ:link-v-K}
\widetilde{H}_n(\link_K v|_J) \cong \widetilde{H}_{n+1}(K|_{J \cup \{v\}}) \oplus \widetilde{H}_n(K|_J), \ \text{for all}\ n\in \Z.
 \end{equation}

    First, we assume that $\left|V\left(\link_K v\right)\right|=m-1$. We prove that $\mu(\link_K v) = \mu(K)$. If $J \subset[m]\backslash \{v\}$ is a minimal non-face of $\operatorname{link}_K v$,  then $\link_K v|_J \cong \partial\Delta^{J}$. So
    $$ \widetilde{H}_n(\link_K v|_J) \cong \begin{cases}
\mathbb{F},  &  \text{if $n=|J|-2$}; \\
 0,  &  \text{otherwise}.
 \end{cases}  $$
  Then by~\eqref{Equ:link-v-K}, either $\widetilde{H}_{|J|-1}(K|_{J \cup \{v\}})\cong \mathbb{F}$ or $ \widetilde{H}_{|J|-2}(K|_J)\cong \mathbb{F}$.  
 So by Lemma~\ref{Lem:Simplex} (b), either $K|_{J \cup \{v\}}$ or $K|_J$ is isomorphic to the boundary of a simplex. Therefore,
  one of $J \cup v$ and $J$ is a minimal non-face of $K$. This implies $\mu(\link_K v) \leqslant \mu(K)$.\n
  
   Conversely, suppose $J \subset[m]$ is a minimal non-face of $K$. Then $K|_J \cong \Delta^{J}$ and 
   \begin{equation} \label{Equ:link-v-K-2}
   \widetilde{H}_n(K|_J) \cong \begin{cases}
\mathbb{F},  &  \text{if $n=|J|-2$}; \\
 0,  &  \text{otherwise}. 
 \end{cases}  
 \end{equation}
  
  There are the following two cases:
  \begin{itemize}
   \item If $v\in J$, then $J\backslash \{v\}$ cannot be
   a face of $\link_K v$ since $J$ is not a face of $K$. In fact, $J\backslash \{v\}$ is a minimal non-face of
  $\link_K v$ since any $J'\subsetneq J$ is a face of $K$.\n
  
  \item If $v \notin J$, it follows from~\eqref{Equ:link-v-K} and~\eqref{Equ:link-v-K-2} that
  $\widetilde{H}_{|J|-2}(\link_K v|_J) \neq 0$.
  Then by Lemma~\ref{Lem:Simplex} (b), $J$ is a minimal non-face of $\link_K v$. 
  \end{itemize}
 
 So we deduce that $\mu(\link_K v) \geqslant \mu(K)$.
 Therefore, $\mu(\link_K v) = \mu(K)$.
 \n
 Moreover, since $\link_K v$ is taut with $m-1$ vertices, by our induction hypothesis,  
 $$\mu(\link_K v)=(m-1)- \mathrm{mdim}(\link_K v)-1= m-\mathrm{mdim}(K)-1.$$ It follows that $\mu(K)=m-\mathrm{mdim} ( K)-1$. \n
 
  Next, we deal with the case $|V(\link_K v)|<m-1$.
   Recall the decomposition~\eqref{Equ:Decomposition}
	  $$ K \cong (v*\link_K v)\cup_{\link_K v}
	     \big(K|_{V(\link_Kv)}*\Delta^{[r_v]}\big).$$
	  Since $K$ and $\link_K v$ are both taut, we have
	  \begin{align*}
	    \widetilde{D}(K) =2^{m-\mathrm{mdim}(K)-1},\ \
	    \widetilde{D}(\link_K v) = 2^{m_v-\mathrm{mdim}(\link_K v)-1}.
	   \end{align*}
	  By Corollary 4.4 (c), we have
	     \begin{equation} \label{Equ:D-tilde-v}
	       \widetilde{D}(K) = 2^{m-m_v-1}  \widetilde{D}(\link_K v) .
	     \end{equation}
	     
  In addition, let 
  $$K'= (v*\link_K v)\cup_{\link_K v}
	    K|_{V(\link_Kv)} = K|_{V(\star_K v)}.$$
	     Then we can easily see that
	    \begin{equation} \label{Equ:mu-K}
	       \mu(K) = \mu(K') + r_v
	       \end{equation}
	    where each vertex $u\in \Delta^{[r_v]}$ contributes
	    a minimal non-face $\{u,v\}$ of $K$. 
	    \n
	    \textbf{Claim-1:} $v\in V_{\mathrm{mdim}}(K')$.
	    \n
	   We do induction on $r_v$. If $r_v=1$, the claim follows immediately from Corollary~\ref{Cor:two-vertex} since $v\in V_{\mathrm{mdim}}(K)$. If $r_v>1$, we take an arbitrary vertex $u\in \Delta^{[r_v]}$ and by Corollary~\ref{Cor:two-vertex} again, $v\in V_{\mathrm{mdim}}(K\backslash u)$. Then since
	   $$ K\backslash u \cong (v*\link_K v)\cup_{\link_K v}
	     \big(K|_{V(\link_Kv)}*\Delta^{[r_v-1]}\big),$$
	    we assert $v\in V_{\mathrm{mdim}}(K')$ by our induction.
	     Claim-1 is proved.\n
	     
	     \textbf{Claim-2:} $\mathrm{mdim} (K)=\mathrm{mdim} (K')$.\n
	     
	 Note that $|V(K')|=m_v+1 = m-r_v$ and $\link_{K'} v=\link_{K} v$. In addition, $K'$ is also taut since it is a full subcomplex of $K$. Then by Claim-1 and~\eqref{Equ:D-tilde-v}, 
	   \begin{equation*} 
	       \widetilde{D}(K') = 2^{(m_v+1)-m_v-1}\widetilde{D}(\link_{K'} v)= \widetilde{D}(\link_{K'} v).
	     \end{equation*}
	  Then we have
	    \begin{align*}
	    2^{m-\mathrm{mdim} (K)-1} =  \widetilde{D}(K) &\overset{\eqref{Equ:D-tilde-v}}{=} 2^{m-m_v-1} \widetilde{D}(\link_K v)\\
	     &=2^{m-m_v-1} \widetilde{D}(\link_{K'} v)=2^{m-m_v-1} \widetilde{D}(K')\\
	     &=2^{m-m_v-1}\cdot 2^{(m_v+1)-\mathrm{mdim}(K')-1}=
	     2^{m-\mathrm{mdim} (K')-1}. 
	     \end{align*}
	    It follows that  $\mathrm{mdim} (K)=\mathrm{mdim} (K')$. Claim-2 is proved.\n
	    
	     Moreover, since 
	    $|V(\link_{K'} v)|= |V(K')|-1$,
	   by the argument of the case $\left|V\left(\link_K v\right)\right|=m-1$, we can deduce that
	   $$\mu(K') = |V(K')| - \mathrm{mdim}(K')-1 = m-r_v-\mathrm{mdim}(K')-1.$$
	   Plugging this into~\eqref{Equ:mu-K}, we obtain 
	   $$\mu(K)=m- \mathrm{mdim}(K')-1 = m- \mathrm{mdim}(K)-1.$$
	   This finishes the proof.
   \end{proof}
 
 \begin{rem}
 Many properties of $K(\mathbf{N})$ proved in~\cite{Anton16} and~\cite{IriKishi18} are parallel to what we have proved for taut simplicial complexes in Section~\ref{Sec-Weakly-Taut} (for example, Theorem~\ref{Thm:Homotopy-Type} is equivalent to \cite[Proposition 4.2]{IriKishi18}). But to build the connection in Theorem~\ref{Thm:Taylor-2}, we need to  prove some of these properties ahead of time for
taut simplicial complexes.
 \end{rem}

 As a application of Theorem~\ref{Thm:Taylor-2}, we can easily classify all the taut simplicial complexes that are Cohen-Macaulay as follows.

  \begin{cor} \label{Cor:Main-1-Repeat}
   A taut simplicial complex
     $K$ is Cohen-Macaulay if and only if the Stanley-Reinser ring $\mathbb{F}[K]$ is a complete intersection.
     \end{cor}
     
      Recall that a graded algebra $\mathbb{F}[x_1,\cdots,x_m]\slash I$ is called a \emph{complete intersection} if the ideal $I$ is generated by a regular sequence.
      It is well known that $\mathbb{F}[K]$ is a complete intersection if and only if all the minimal non-faces of $K$ are pairwise disjoint, which is equivalent to saying that $K$ is isomorphic to $ \partial \Delta^{[n_1]}*\cdots*\partial \Delta^{[n_k]}$ or $\Delta^{[r]} * \partial \Delta^{[n_1]}*\cdots*\partial \Delta^{[n_k]}$ for some positive integers $n_1,\cdots, n_k$ and $r$. 
So the above corollary classifies all the taut Cohen-Macaulay complexes.    
      
 \begin{proof}  
   The proof of the ``if'' part is trivial.  
  	For the ``only if'' part, if $K$ is taut and Cohen-Macaulay, we prove that all the minimal non-faces of $K$ are pairwise disjoint. First of all, according to Definition~\ref{Def:Taut}, $K$ being taut implies that the projective dimension of $\F[K]$ is
$
\operatorname{pdim}(\mathbb F[K])
= m-\mdim(K)-1$.
Then by the Auslander--Buchsbaum theorem, the depth of
$\mathbb F[K]$ is
\[
\operatorname{depth}(\mathbb F[K])=m- \operatorname{pdim}(\mathbb F[K])= \mdim(K)+1.
\]
Moreover, the Krull dimension of $\mathbb{F}[K]$ is $\dim(K)+1$. So the assumption that
  $K$ is Cohen-Macaulay implies that   $\mdim(K)=\dim(K)$, i.e. $K$ is pure.
  	\n
  	Moreover, by Theorem~\ref{Thm:Taylor-2}, $K$ has a minimal Taylor resolution. So according to Theorem~\ref{Thm:IriKishi},
  	$K\cong K(\mathbf{N})$ for some sequence of subsets
  	$\mathbf{N}=\{ N_1,\cdots, N_r\}$ of $[m]\backslash \mathbf{a}$ where $m=|V(K)|$ and $\mathbf{a}=\{v_1,\cdots, v_r\}$. Note that $\mathbf{a}^c=[m]\backslash \mathbf{a}$ is a maximal simplex of $K(\mathbf{N})$. Then since $K$ is pure,
  	$\mathrm{dim}(K)=m-r-1$. If $N_i\cap N_j\neq \varnothing$ for some $i\neq j$, then for any vertex
  	$w \in N_i\cap N_j$, $(\mathbf{a}^c\backslash \{w\}) \cup \{v_i,v_j\} $ is also a maximal simplex of $K(\mathbf{N})$.
  	But the dimension of $(\mathbf{a}^c\backslash \{w\}) \cup \{v_i,v_j\} $ is $m-r$, contradiction. So
  	all the minimal non-faces $\widetilde{N}_1, \cdots, \widetilde{N}_r$ of $K(\mathbf{N})$ are pairwise disjoint.  	This finishes the proof. 
  \end{proof}
  
\vskip .4cm

\section{Taut simplicial complexes that are shifted} \label{Sec:Shifted}

 A simplicial complex $K$ with $m$ vertices 
  	is called \emph{shifted} if there exists a labeling of
  	the vertices of $K$ as $\{v_1,\cdots, v_m\}$ such that: for any face $\sigma$ of $K$ and any $v_i\in\sigma$, $(\sigma\backslash\{v_i\}) \cup \{v_j\}$ is also a face of $K$ for every $j<i$.  In other words, one can always swap a vertex with larger index for one with smaller index and stay inside the complex.  Note that the labeling of $V(K)$ by $1$ through $m$ is equivalent to defining a total order on $V(K)$. 
   The following lemma gives another way to understand a shifted simplicial complex via its minimal non-faces.
  	
  \begin{lem}\label{Lem:Shift-Equiv}
   A simplicial complex $K$ with $m$ vertices 
  	is shifted if and only if there exists a labeling of
  	the vertices of $K$ as $\{v_1,\cdots, v_m\}$ such that for any minimal non-face $J$ of $K$ and any $v_j\in  J$,  
  	$(J\backslash\{v_j\}) \cup \{v_i\}$ is not a face of $K$ for every $v_i\notin J$ with $i>j$.
  \end{lem}
  \begin{proof}
   The necessity is obvious. For the sufficiency, 
  suppose we have such a labeling of the vertices of $K$
  as described in the lemma. Assume that there is a  face $\sigma$ of $K$ and some vertex
  $v_i\in \sigma$ and $j<i$, but
   $(\sigma\backslash\{v_i\}) \cup \{v_j\}$ is not a face of $K$. Let $J$ be a minimal non-face contained in  $(\sigma\backslash\{v_i\}) \cup \{v_j\}$. We must have $v_j\in J$ since otherwise $J\subseteq \sigma\backslash \{v_i\} \subseteq \sigma$ which contradicts that $\sigma$ is face of $K$. 
  But by our assumption, $(J\backslash\{v_j\}) \cup \{v_i\}$
  is not a face of $K$, which implies that $\sigma$ is not a face either, contradiction. So by definition, $K$ is shifted.
  \end{proof}
  \nn

   For the complex $K(\mathbf{N})$ defined in Definition~\ref{Defi:KN},
 we can judge when $K(\mathbf{N})$ is a shifted complex by the following theorem.

 \begin{thm} \label{Thm:Taut-Shifted-Complex}
   The complex $K(\mathbf{N})$ defined by $\mathbf{N}=\left\{N_1, \cdots, N_r\right\}$ is shifted if and only if
 $N_1, \cdots, N_r$ can be ordered to be a monotone sequence with respect to set inclusion.  
   \end{thm}
   \begin{proof}
     For the sufficiency, assume that $N_1\supseteq N_2\supseteq \cdots \supseteq N_r$. Up to relabeling of the vertices of $K(\mathbf{N})$, we can assume
     $$N_1 =\{v_{p_1},\cdots, v_m\}, \ \ N_2=
     \{v_{p_2},\cdots, v_m\}, \ \
     N_r= \{v_{p_r},\cdots, v_m\} $$
     where $r<p_1 \leq p_2 \leq \cdots \leq p_r \leq m$.
     Note that $ v_{r+1},\cdots, v_{p_1-1} $ are all the vertices not in $\widetilde{N}_1 \cup \cdots \cup \widetilde{N}_r $. Next, define a new order $\prec$ of the vertices of $K(\mathbf{N})$ by:
     $$ v_{r+1} \prec \cdots \prec v_{p_1-1} \prec \underset{\widetilde{N}_1 \cup \{v_2,\cdots,v_r\}}{\underbrace{ v_1 \prec v_{p_1} \prec \cdots \prec \underset{\widetilde{N}_2 \cup \{v_3,\cdots,v_r\}}{\underbrace{v_2 \prec v_{p_2} \prec \cdots \prec  \underset{ \scalebox{1.2}{$\cdots$} \ \widetilde{N}_r\quad \  \ \, }{\underbrace{v_r \prec v_{p_r} \prec \cdots \prec v_m}}}}}},$$ 
     where if $p_i=p_{i+1}$, let $v_i \prec v_{i+1} \prec v_{p_i}=v_{p_{i+1}}$. 
    If we relabel the vertices of $K(\mathbf{N})$ by $1$ through $m$ with respect to $\prec$, it is easy to check that the minimal non-faces $\widetilde{N}_1,\cdots, \widetilde{N}_r$ of $K(\mathbf{N})$ satisfy the condition in Lemma~\ref{Lem:Shift-Equiv}. Indeed, if there is a vertex $v'\notin \widetilde{N}_i$ and a vertex $v\in \widetilde{N}_i$ with $v\prec v'$, then $v'$ can only come from $\{v_{i+1},\cdots, v_{r}\}$. So $(\widetilde{N}_i\backslash \{v\}) \cup \{v'\}$ contains at least one of 
    $\widetilde{N}_{i+1},\cdots, \widetilde{N}_{r}$, which implies that $(\widetilde{N}_i\backslash \{v\}) \cup \{v'\}$ is not a face of $K(\mathbf{N})$. Then it follows from
    Lemma~\ref{Lem:Shift-Equiv} that $K(\mathbf{N})$ is shifted.\n
     
     To prove the necessity, assume that $K(\mathbf{N})$ is a shifted complex with respect to a labeling $\{v_1,\cdots, v_m\}$ of its vertices. \n
     
     \textbf{Claim:} There exists a subset $N_i$ which contains every $N_j$, $1\leq j\neq i \leq r$.\n
     
      Indeed, if no such a $N_i$ exists, 
     then we have a vertex
      $v_{a_i} \in (\bigcup_{j\neq i} N_j)\backslash N_i$ for every $1\leq i \leq r$. 
         Without loss of generality, assume 
      $a_1 \leq \cdots\leq a_r$. Then $v_{a_1}$ is contained in some $N_i$ where $2\leq i \leq r$. Clearly, $v_{a_i}\neq v_{a_1}$ and so $a_1 < a_i$ by our assumption.
      By Lemma~\ref{Lem:Shift-Equiv}, the set
     $(\widetilde{N}_i\backslash \{v_{a_1}\})\cup \{v_{a_i}\}$ is not a face of $K(\mathbf{N})$.
  But $(\widetilde{N}_i\backslash \{v_{a_1}\})\cup \{v_{a_i}\}$ does not 
  contain any minimal non-face $\widetilde{N}_1,\cdots, \widetilde{N}_r$, which is a contradiction. So the claim
  is proved.\n

   By the above claim, there is some $1\leq i_1\leq r$ such that $N_{i_1}$ contains every $N_j$, $1\leq j\neq i_1 \leq r$. Moreover, we can apply the same argument to
   the collection $\{N_1,\cdots, N_{i_1-1}, N_{i_1+1},\cdots, N_r\}$ and deduce that there also exists a set $N_{i_2}$ that contains all other sets
   in this collection. So by iterating the above argument, we can order $N_1,\cdots, N_r$ to be 
     a monotonic sequence $N_{i_1}\supseteq N_{i_2}\supseteq \cdots \supseteq N_{i_r}$.
     This finishes the proof of the theorem.
       \end{proof}

\vskip .4cm

\section{Simplicial wedge preserves tautness} \label{Sec:Wedge}

In this section, we prove that except the join, there is another type of operation on simplicial complexes called
 \emph{simplicial wedge} which also preserves tautness. This operation was originally introduced by Provan and Billera~\cite[p.\,578]{ProBil80}.

\begin{defi}[Simplicial Wedge]
Suppose $K$ is a simplicial complex with $m$ vertices. Choose a fixed vertex $v$ of $K$ and define a new simplicial complex $K(v)$ with $m+1$ vertices
\begin{equation} \label{Equ:K-v}
K(v):= ( \overline{v_1 v_2} * \link_K v ) \cup
  \big(( v_1 \cup v_2) * (K \setminus v) \big)
\end{equation} 
where $\overline{v_1 v_2}$ is a $1$-simplex spanned by two new vertices $v_1$ and $v_2$. The $K(v)$ is called the \emph{simplicial wedge of $K$ on $v$}.
 \end{defi}
 
 \begin{lem} \label{Lem:Wedge-Equal}
  Let $K$ be a simplicial complex. For any vertex $v$ of $K$, 
  $$\widetilde{D}(K(v))=\widetilde{D}(K).$$
 \end{lem}
 \begin{proof}
  By definition, $V(K(v))=V(K\setminus v) \cup \{v_1,v_2\}$. Observe that 
  \begin{itemize}
  \item If $J\subseteq V(K\setminus v)$,
  $K(v)|_J = K|_J$.\n

  \item If $J\subseteq V(K\setminus v)$,
  $\|K(v)|_{J\cup\{v_1, v_2\}}\|$ is homeomorphic to the suspension
  of $\|K|_{J\cup\{v\}}\|$. 
   So
    $\widetilde{tb}(K(v)|_{J\cup\{v_1, v_2\}})=\widetilde{tb} (K|_{J\cup\{v\}})$.\n
   
   \item If $J\subseteq V(K(v))$ contains $v_1$ or $v_2$ but not both, then $K(v)|_J$ is contractible. 
    
   \end{itemize}
  \n
  So by the Hochster's formula~\eqref{Equ:D-tilde-K-J} and the definition of $K(v)$, we obtain 
   \begin{align*}
     \widetilde{D}(K(v)) &=\sum_{J\subseteq V(K\setminus v)} \widetilde{tb}(K(v)|_J) +
     \sum_{J\subseteq V(K\setminus v)} \widetilde{tb}(K(v)|_{J\cup\{v_1, v_2\}}) \\
         &=\sum_{J\subseteq V(K\setminus v)} \widetilde{tb}(K|_J) +
     \sum_{J\subseteq V(K\setminus v)} \widetilde{tb}(K|_{J\cup \{v\}}) = \widetilde{D}(K).
   \end{align*}
 
  Another way to prove this lemma is using the result of Bahri-Bendersky-Cohen-Gitler~\cite[Corollary 7.6]{BBCG15} which says that
  $\mathcal{Z}_{K(v)}$ and $\mathcal{Z}_K$ have isomorphic ungraded cohomology rings. So by~\eqref{Equ:DK-ZK}, 
   $\widetilde{D}(K(v)) = tb(\mathcal{Z}_{K(v)})=tb(\mathcal{Z}_{K})=\widetilde{D}(K)$.
 \end{proof}

 \begin{thm} \label{Thm:wedge-Weak-Taut}
   Let $v$ be a vertex of a simplicial complex $K$. Then $K(v)$ is taut if and only if $K$ is taut. 
 \end{thm}
 \begin{proof}
    Let $\xi$ be a maximal simplex of $K$. If $\xi$ contains $v$, then from~\eqref{Equ:K-v} we can see that $\overline{v_1 v_2}* (\xi \setminus v) $ is a maximal simplex in $K(v)$.
  If $\xi$ does not contain $v$, then $v_1*\xi$ and
  $v_2*\xi$ are both maximal simplices of $K(v)$. Conversely, it is easy to see that any maximal simplex of $K(v)$ can be written in the form $\overline{v_1 v_2}* (\xi \setminus v)$, 
  $v_1*\xi$ or $v_2*\xi$ for some maximal simplex $\xi$ of $K$. This implies 
  $$\mdim(K(v)) =\mdim(K)+1.$$
   Then by Lemma~\ref{Lem:Wedge-Equal} and the fact that $|V(K(v))| = |V(K)|+1$, we deduce that
 $$\text{  $\widetilde{D}(K)=2^{|V(K)|-\mdim(K)-1}$ if and only if 
   $\widetilde{D}(K(v))=2^{|V(K(v))|-\mdim(K(v))-1}$.}$$
   The theorem is proved. 
 \end{proof}

  Moreover generally, for a simplicial complex $K$ with vertex set $\left\{v_1, v_2, \ldots, v_m\right\}$ and a sequence of positive integers $\mathcal{J}=\left(j_1, j_2, \ldots, j_m \right)$, one can define a new simplicial complex $K(\mathcal{J})$ on $j_1+j_2+\cdots+j_m$ new vertices, labeled
$$
v_{11}, v_{12}, \ldots, v_{1 j_1}, v_{21}, v_{22}, \ldots, v_{2 j_2}, \ldots, v_{m 1}, v_{m 2}, \ldots, v_{m j_m},
$$
with the property that
$$
\{v_{i_1 1}, v_{i_1 2}, \ldots, v_{i_1 j_{i_1}}, v_{i_2 1}, v_{i_2 2}, \ldots, v_{i_2 j_{i_2}}, \ldots, v_{i_k 1}, v_{i_k 2}, \ldots, v_{i_k j_{i_k}}\}
$$
is a minimal non-face of $K(\mathcal{J})$ if and only if $\left\{v_{i_1}, v_{i_2}, \ldots, v_{i_k}\right\}$ is a minimal non-face of $K$. Moreover, all minimal non-faces of $K(\mathcal{J})$ have this form. The above definition 
of $K(\mathcal{J})$ was introduced in~\cite[Definition 2.1]{BBCG15}. In fact, $K(\mathcal{J})$ can also be produced  
     by iteratively applying the simplicial wedge construction to $K$. So by Theorem~\ref{Thm:wedge-Weak-Taut}, we obtain the following corollary immediately.
     
     \begin{cor}
   Let $K$ be a simplicial complex with $m$ vertices.
   For a  sequence of positive integers  $\mathcal{J}=(j_1,\cdots, j_m)$, the simplicial wedge $K(\mathcal{J})$ is taut if and only if $K$ is taut.
     \end{cor}

 \vskip .2cm

  \section*{Acknowledgment}
   The authors want to thank Professor Daisuke Kishimoto
 for some helpful suggestion, and thank Qiyuan Gu for pointing out an error in the proof of Theorem~\ref{Thm:Homotopy-Type} in the earlier version of this paper.

\end{document}